\documentclass[]{sn-jnl}

\usepackage{graphicx}%
\usepackage{multirow}%
\usepackage{amsmath,amssymb,amsfonts}%
\usepackage{amsthm}%
\numberwithin{equation}{section}
\usepackage{changes}
\usepackage{tikz}
\usetikzlibrary{shapes.geometric, arrows.meta, positioning}
\setdeletedmarkup{}
\usepackage{mathrsfs}%
\usepackage[title]{appendix}%
\usepackage{xcolor}%
\usepackage{textcomp}%
\usepackage{manyfoot}%
\usepackage{booktabs}%
\usepackage{algorithm}%
\usepackage{algorithmicx}%
\usepackage{algpseudocode}%
\usepackage{upgreek}
\usepackage{listings}%
\usepackage{lineno} 
\usepackage{comment}
\usepackage{geometry}
\usepackage{tikz}
\usepackage{caption}
\usepackage{adjustbox}
\usetikzlibrary{shapes.geometric, arrows}
\tikzstyle{startstop} = [rectangle, rounded corners, 
minimum width=3cm, 
minimum height=1cm,
text centered, 
draw=black, 
fill=red!30]

\tikzstyle{io} = [trapezium, 
trapezium stretches=true, 
trapezium left angle=70, 
trapezium right angle=110, 
minimum width=3cm, 
minimum height=1cm, text centered, 
draw=black, fill=blue!30]

\tikzstyle{process} = [rectangle, 
minimum width=3cm, 
minimum height=1cm, 
text centered, 
text width=3cm, 
draw=black, 
fill=orange!30]

\tikzstyle{decision} = [diamond, 
minimum width=3cm, 
minimum height=2cm,
text centered, 
draw=black, 
fill=green!30]
\tikzstyle{arrow} = [thick,->,>=stealth]
\usepackage[style=numeric, backend=biber]{biblatex}
\theoremstyle{thmstyleone}%
\newtheorem{theorem}{Theorem}[section]
\newtheorem{lemma}[theorem]{Lemma}
\newtheorem{corollary}[theorem]{Corollary}
\newtheorem{proposition}[theorem]{Proposition}
\newtheorem{definition}[theorem]{Definition}

\newtheorem{condition}{$\mathcal{B}$}

\DeclareMathOperator*{\Limsup}{Lim\,sup}

\theoremstyle{thmstyletwo}%

\theoremstyle{thmstylethree}%

\begin{document}


\title[Article Title]{Nonmonotone Coderivative-Based Newton Methods for Nonsmooth Optimization with Machine Learning Applications}

\author[1]{\fnm{Sina} \sur{Kazemdehbashi}}\email{hm3369@wayne.edu}
\author[1]{\fnm{Yanchao} \sur{Liu}}\email{yanchaoliu@wayne.edu}
\author[2]{\fnm{Boris S.} \sur{Mordukhovich}}\email{aa1086@wayne.edu}

\affil[1]{\orgname{Department of Industrial \& Systems Engineering, Wayne State University}, \orgaddress{\city{Detroit}, \state{Michigan}, \country{USA}}}
\affil[2]{\orgname{Department of Mathematics \textcolor{black}{and the Institute for AI and Data Science}, Wayne State University}, \orgaddress{\city{Detroit}, \state{Michigan}, \country{USA}}}

\abstract{

Newton-type algorithms are among the most effective methods for solving optimization problems because of their rapid local convergence. However, extending these methods to nonsmooth optimization is challenging due to the difficulty of incorporating second-order information. To address this issue, we propose a coderivative-based Newton framework for solving both unconstrained and constrained nonsmooth optimization problems using tools from variational analysis and generalized differentiation. The proposed method employs generalized Hessians, defined as coderivatives of the subgradient mapping, and is applicable to both $C
^{1,1}$ functions and convex composite optimization problems with extended-real-valued components. To enhance robustness, the algorithm incorporates an adaptive regularization strategy that enables it to handle problems whose generalized Hessians are positive semidefinite. In addition, a Hybrid Adaptive Nonmonotone (HAN) line search scheme is developed as a globalization technique to improve global convergence and practical performance. Under standard assumptions, both the exact and inexact versions of the algorithm are globally convergent and achieve local superlinear convergence when the associated subgradient mapping satisfies the semismooth* property. The proposed framework is further extended to convex composite optimization through the forward-backward envelope, allowing it to handle problems with or without strong convexity in the smooth component of the objective function. Numerical experiments on Lasso, logistic Lasso, and support vector machine (SVM) problems demonstrate the efficiency and robustness of the proposed algorithm. Comparisons with several well-established first-order and second-order methods for nonsmooth optimization show that the proposed approach is computationally competitive while maintaining strong convergence properties.

}


\maketitle

\section{Introduction}

The classical Newton method is a well-known and efficient algorithm for solving unconstrained optimization problems in which the objective function belongs to the class \( C^2 \) and is of the form
\begin{equation}\label{eq:unconstrained_main_problem}
\min f(x) \quad \text{s.t.} \quad x \in \mathbb{R}^n ,
\end{equation}
provided that the Hessian matrix \( \nabla^2 f(\bar{x}) \) is positive definite in a neighborhood of the local solution \( \bar{x} \). Although the Newton algorithm is recognized as a fast method with quadratic or superlinear convergence rates, its convergence is guaranteed only when the initial point is sufficiently close to the local solution. This limitation, along with the high computational cost and limited applicability, is among the main drawbacks of this method. Accordingly, several strategies—such as Armijo line search, trust-region methods, and cubic regularization—are used to design improved Newton methods that achieve global convergence while preserving quadratic or superlinear local convergence rates. In application areas such as machine learning, some unconstrained optimization problems arise whose objective functions are not twice continuously differentiable, making the classical Newton method inapplicable.
To address this challenge, several generalized Newton methods have been proposed in the literature that are applicable to functions in the class \( C^{1,1} \), meaning that the objective function is continuously differentiable with locally Lipschitz continuous gradients (second-order nonsmooth). Among these methods, the \emph{semismooth Newton method} (SNM) is particularly popular. It relies on Clarke’s generalized Jacobian \cite{clarke2023optimization}, and there is a rich body of literature on this method and its modifications. For further details and background, see the books \cite{facchinei2003finite,izmailov2014newton,klatte2002nonsmooth} and the references therein. Further Newton-type methods addressing nonsmooth equations, generalized equations, and optimization and variational problems are presented in \cite{facchinei2003finite, bonnans1994local,hoheisel2012generalized,izmailov2014newton,klatte2002nonsmooth, nesterov2018lectures, robinson1994newton}.

This paper contributes to a recent line of research in variational analysis through the investigation of nonsmooth Newton methods whose algorithmic frameworks rely on Mordukhovich coderivatives \cite{mordukhovich2009variational} rather than Clarke generalized Jacobians. The second-order construction of variational analysis is equipped with comprehensive calculus rules and admit explicit representations in terms of the problem data for broad classes of nonsmooth functions arising in optimization and its applications, including machine learning, data science, statistics, biochemical modeling, and many other areas; see \cite{mordukhovich2024second} for additional details and references. The literature contains several second-order algorithms \cite{khanh2024globally,chao2024coderivative,duy2023generalized,mordukhovich2021generalized} based on the Mordukhovich coderivative of the gradient or generalized Hessian.

In \cite{khanh2024globally}, the authors proposed a coderivative-based Newton method in which the Armijo line search is employed as the globalization strategy for \( C^{1,1} \) functions, with applications to composite convex optimization problems. The main advantage of the Armijo line search is its low computational cost; however, a major drawback is that it may select very small step sizes. To address this limitation, \cite{chao2024coderivative} proposed a coderivative-based Newton method that uses the Wolfe line search as the globalization strategy. Although this approach can yield larger step sizes, it incurs higher computational cost compared to the Armijo line search, since the Wolfe line search requires evaluating an additional curvature condition involving the gradient of the function.

To overcome the drawbacks of the Armijo and Wolfe line searches, we propose a Hybrid Adaptive Nonmonotone (HAN) line search that introduces a step-size-dependent switching mechanism. Unlike traditional nonmonotone methods, which adapt memory gradually or apply curvature conditions uniformly, the HAN line search selectively increases the nonmonotone memory and activates curvature checks only when the trial step size is small, while allowing large steps that satisfy the nonmonotone Armijo condition to be accepted directly without curvature evaluation. This design reduces the computational cost associated with curvature evaluations while avoiding excessively small step sizes. Based on the proposed line search, we develop a nonmonotone coderivative-based Newton method in which the HAN line search is employed as the globalization strategy. For $C^{1,1}$ function, we establish global convergence and a superlinear local convergence rate for the proposed method to a tilt-stable minimizer as defined in Section \ref{sec:tools_variational}.

\deleted{line search methods adjust only the step size along a prescribed search direction and do not modify the direction itself. In contrast, trust region methods are a globalization strategy that simultaneously control both the step size and the search direction. Moreover, trust region methods can achieve convergence even under indefinite conditions and do not require the search direction to be a descent direction. These methods are based on constructing a local approximation of the objective function within a neighborhood (the trust region) around the current iterate. Trust region Newton-type methods have been extensively studied in the literature [2,2,2,2,2], and trust region variants of nonsmooth second-order algorithms have also been investigated in [2,2,2]. In this work, we propose a trust region coderivative-based Newton algorithm in which, compared to existing approaches, the generalized Hessian is used instead of the Clarke Jacobian, and a trust region framework is employed as the globalization strategy. We establish global convergence as well as a locally superlinear convergence rate for $C^{1,1}$ functions. Overall, by explicitly restricting the region where the quadratic model is trusted, trust region methods offer greater robustness to indefinite or ill-conditioned Hessians and provide stronger theoretical guarantees than line search approaches, particularly for nonconvex optimization problems.}

In contrast to the constant regularization strategies adopted in the recent works \cite{khanh2024globally,chao2024coderivative}, we employ two adaptive regularization schemes. The first scheme is motivated by trust-region methods: it compares the reduction predicted by the second-order model with the actual reduction in the objective function and adjusts the regularization parameter accordingly, decreasing it when the model accurately predicts the objective reduction and increasing it otherwise. The second scheme adapts the regularization parameter according to the step length produced by the HAN line search, decreasing the parameter whenever a unit step size is accepted and increasing it otherwise. We establish that both adaptive regularization strategies retain the global convergence and local superlinear convergence guarantees of the proposed algorithm. Furthermore, numerical experiments demonstrate that these strategies enhance computational efficiency and improve practical performance.

In this work, we apply the proposed algorithm to the following convex composite minimization problem:
\begin{equation*}
    \min_{x \in \mathbb{R}^n} \; h(x) + g(x),
\end{equation*}
where $h$ is a smooth convex function and $g : \mathbb{R}^n \to \bar{\mathbb{R}} := (-\infty, \infty]$ is a lower semicontinuous extended real-valued convex function. We reformulate this problem as a $C^{1,1}$ optimization problem using the Forward--Backward Envelope (FBE), which enables the application of the proposed algorithm. By specializing $h$ to a quadratic function, we further apply our approach to solve the Lasso problem. Originally introduced by Tibshirani~\cite{tibshirani1996regression}, the Lasso has widespread applications in statistics, machine learning, and artificial intelligence.
Compare to the proposed algorithm in \cite{khanh2024globally,chao2024coderivative} our algorithm is different in the aspect of the globalization stategy and regularization technique. Next, we apply the proposed method to the Support Vector Machine (SVM) problem. Introduced by Cortes and Vapnik \cite{cortes1995support}, SVM is one of the most widely used supervised learning models for classification and regression. Finally, we employ the proposed algorithm as the inner solver within a proximal Newton framework and develop a coderivative-based proximal Newton (CPN) method to handle the general case in which $h$ is not quadratic. The effectiveness of the proposed CPN method is demonstrated on the Logistic Lasso problem.

 Finally, we summarize the main contributions of this paper as follows. First, we propose a nonmonotone coderivative-based Newton method equipped with a Hybrid Adaptive Nonmonotone (HAN) line search and an adaptive regularization strategy. Second, under suitable assumptions, we extend existing convergence results for coderivative-based Newton methods to the setting of inexact second-order information. Third, we present numerical experiments on convex composite optimization problems with quadratic smooth terms arising in well-known machine learning models, including the Lasso and Support Vector Machine (SVM). Also, we develop a coderivative-based proximal Newton (CPN) method to handle the case where the smooth component is non-quadratic, and demonstrate its performance on the Logistic Lasso problem.

The rest of the paper is organized as follows. In Section~\ref{sec:tools_variational}, we review some preliminary concepts and tools from variational analysis that will be used throughout the paper. In Section~\ref{sec:nonmonotone_coderivative_algorithm}, we propose the NCN method for solving optimization problem with $C^{1,1}$ objective function, using the HAN line search as a globalization strategy. We also describe the adaptive regularization mechanism and the HAN line search procedure incorporated into the proposed method. In Section~\ref{sec:convergence_analysis}, we establish a general convergence framework for coderivative-based Newton methods. In particular, we prove global convergence and superlinear local convergence rate under both exact and inexact second-order information. Section~\ref{sec:application_lasso} presents an application of the proposed NCN method to convex composite optimization problems with a quadratic smooth component through the forward-backward envelope. Within this framework, we consider the Lasso and SVM problems as representative applications arising in machine learning. In addition, we propose a CPN method for convex composite optimization problems whose smooth component is nonquadratic; also, we consider logistic lasso as an well-known exnample in this framework. In Section~\ref{sec:experiments}, we present numerical experiments to evaluate the performance of the NCN method on Lasso and SVM problems and compare it with several well-established first- and second-order methods from the literature. We also investigate the numerical performance of the CPN method on the logistic Lasso problem and compare it with existing methods. Finally, Section~\ref{sec:conclusion} summarizes the main theoretical results and numerical findings of the paper and outlines several directions for future research.

\section{Tools of Variational Analysis}\label{sec:tools_variational}
At the beginning of this section, we review some preliminary notation. Throughout this work, $\mathbb{B}$ denotes the closed unit ball, $\mathbb{B}_r(x):=x+r\mathbb{B}$ denotes the closed ball of radius $r>0$ centered at $x$, and $\mathbb{N}:={1,2,3,\dots}$ denotes the set of natural numbers.

Given a set-valued mapping $F:\mathbb{R}^n \rightrightarrows \mathbb{R}^m$ between finite-dimensional spaces, its \textit{outer limit} (Painlevé-Kuratowski) at $\bar{x}$ is defined as
\begin{equation}
    \Limsup_{x \to \bar{x}} F(x):=\{y \in \mathbb{R}^n \ |\ \exists \ \textnormal{sequences } x_k \to \bar{x},\  y_k \to y\ \  \text{with} \ \ y_k \in F(x_k),\  k\in \mathbb{N} \}.
\end{equation}

In the following, we review some useful tools from variational analysis that will be employed in this work, see \cite{mordukhovich2018variational,rockafellar1998variational} for additional materials. Given $\Omega \subset \mathbb{R}^s$, the (Fr\'echet) \textit{regular normal cone} to $\Omega$ at $\bar{x}\in \Omega$ is defined as: 
\begin{equation}
    \widehat{N}_{\Omega}(\bar{z}):=\{v \in \mathbb{R}^s \ |\  \limsup_{z\overset{\Omega}{\rightarrow} \bar{z}} \frac{\langle v,z-\bar{z}\rangle}{\|z-\bar{z}\|}\le0\},
\end{equation}
where $z \overset{\Omega}{\rightarrow}\bar{z}$ means that $z \to \bar{z}$ with $z \in \Omega$. Also, the (Mordukhovich) \textit{limiting normal cone} to set $\Omega$ at $\bar{z}\in \Omega$ is defined as:
\begin{equation}\label{eq:limiting_normal_cone}
    N_{\Omega}(\bar{z}):= \Limsup_{z \overset{\Omega}{\rightarrow}\bar{z}} \widehat{N}_{\Omega}(z).
\end{equation}

Let $f:\mathbb{R}^n\to\overline{\mathbb{R}}$ be an extended-real-valued function where
$\mathrm{dom}\, f:=\{x\in \mathbb{R}^n \ |\ f(x)<\infty\}$ and $\mathrm{epi}\,f:=\{(x,y) \in \mathbb{R}^{n+1}\ | \ y\ge f(x)\}$. Then, the (limiting) \textit{subdifferential} of $f$ at $\bar{x} \in \mathrm{dom}\, f$ is defined as:
\begin{equation}\label{eq:limiting_subdifferential_definition}
    \partial f(\bar{x}):=\{v \in \mathbb{R}^n \ |\ (v,-1)\in N_{\mathrm{epi}\, f}(\bar{x},f(\bar{x}))\},
\end{equation}
where $N_{\mathrm{epi}\, f}(\bar{x},f(\bar{x}))$ is defined in (\ref{eq:limiting_normal_cone}).

For a set-valued mapping (multifunction) $F:\mathbb{R}^n\rightrightarrows\mathbb{R}^m$, its graph is defined as
$
\operatorname{gph}F
:=
\{(x,y)\in\mathbb{R}^n\times\mathbb{R}^m\mid y\in F(x)\}.
$
The \textit{limiting coderivative} of $F$ at $(\bar{x},\bar{y})\in\operatorname{gph}F$ is the set-valued mapping
$D^\ast F(\bar{x},\bar{y}):\mathbb{R}^m\rightrightarrows\mathbb{R}^n$ defined by
\begin{equation}\label{eq:limiting_coderivative}
D^\ast F(\bar{x},\bar{y})(u)
:=
\big\{
v\in\mathbb{R}^n
\mid
(v,-u)\in N_{\operatorname{gph}F}(\bar{x},\bar{y})
\big\},
\qquad u\in\mathbb{R}^m,
\end{equation}
where $N_{\operatorname{gph}F}(\bar{x},\bar{y})$ denotes the limiting normal cone to $\operatorname{gph}F$ at $(\bar{x},\bar{y})$. In the case where $F=f$ is single-valued, we omit $\bar{y}$ from the notation in (\ref{eq:limiting_coderivative}). The following scalarization formula for a single-valued map $f:\mathbb{R}^n \to \mathbb{R}^m$ is useful for calculating its limiting coderivative.
\begin{equation}
    D^\ast f(\bar{x})(u)=\partial\langle u,f \rangle (\bar{x}) \ \ \ \text{for all}\ \ u\in \mathbb{R}^m.
\end{equation}

Given a single-valued mapping $f:\mathbb{R}^n \to \mathbb{R}^m$ that is locally Lipschitz continuous around $\bar{x}$, the (Clarke) \textit{generalized Jacobian} of $f$ at $\bar{x}$ is defined as the convex hull of the set of limiting Jacobian matrices, as follows:
\begin{equation}\label{eq:clark_jacobian}
    \overline{\nabla}f(\bar{x}):=\{\lim_{k\to \infty} \nabla f(x_k) \ |\ x_k \to \bar{x},\  x_k\in \Omega_f\},
\end{equation}
where $\Omega_f$ denotes the set of points at which $f$ is differentiable. The nonemptiness and compactness of (\ref{eq:clark_jacobian}) in $\mathbb{R}^{m \times n}$ follow from the classical Rademacher theorem. The following equality illustrates the relationship between the Clarke generalized Jacobian and the limiting coderivative:
\begin{equation}\label{eq:relatioship_coderivative_clark}
    \text{co}D^\ast f(\bar{x})(u)=\{A^{\ast}u \ |\ A\in \text{co}\overline{\nabla}f(\bar{x})\}, \ \ \ u \in \mathbb{R}^m.
\end{equation}

Next, let $f: \mathbb{R}^n \to \mathbb{R}^m$ be locally Lipschitz around $\bar{x}$. We say that $f$ is semismooth at $\bar{x}$ if it is directionally differentiable at $\bar{x}$ and satisfies the following estimate:
\begin{equation}
    f(x)-f(\bar{x})-A(x-\bar{x})=o(\|x-\bar{x}\|),
\end{equation}
when $x\to \bar{x}$ and $A \in \text{co} \overline{\nabla}f(x)$. This subclass of single-valued, locally Lipschitz mappings is particularly important in numerical optimization. In addition, we recall a more general concept of semismooth property. Given a set-valued mapping $F:\mathbb{R}^n \rightrightarrows \mathbb{R}^m$ is $\textit{semismooth}^\ast$ at $(\bar{x},\bar{y})\in \textnormal{gph} F$ if whenever $(u,v)\in \mathbb{R}^n \times \mathbb{R}^m$ we have 
\begin{equation}
    \langle u^\ast ,u \rangle=\langle v^\ast,v \rangle \ \ for \ all \  (v^\ast,u^\ast)\in \textnormal{gph}D^\ast F((\bar{x},\bar{y});(u,v)).
\end{equation}

Now, we recall several notions from variational analysis that will be used in establishing the superlinear convergence results presented in Section~\ref{sec:convergence_analysis}. These notions were introduced by Rockafellar and are presented in \cite{rockafellar1998variational}. A set-valued mapping $S:\mathbb{R}^n \rightrightarrows \mathbb{R}^m$ is proto-differentiable at $(\bar{x},\bar{y})\in \textnormal{gph} S$ if for any $\bar{\omega}\in \mathbb{R}^n$, $\bar{z}\in \text{Limsup}_{t \downarrow0,\omega \to \bar{\omega}}(S(\bar{x}+t \omega)-\bar{y})/t$, and $t_k \downarrow0$ there exists $\omega_k \rightarrow\bar{\omega}$ and $z_k \to \bar{z}$ such that $z_k\in (S(\bar{x}+t_k \omega_k))/t_k$ whenever $k \in \mathbb{N}$. Given $\varphi:\mathbb{R}^n \to\mathbb{\bar{R}}$ with $\bar{x}\in \textnormal{dom} \varphi$, consider the family of second-order finite difference 
\[
\Delta_{\tau}^2\varphi(\bar{x},v)(u):=\frac{\varphi(\bar{x}+\tau u)-\varphi(\bar{x})-\tau\langle v,u\rangle}{\frac{1}{2}\tau^2}
\]
and define the second subderivative of $\varphi$ at $\bar{x}$ for $v \in \mathbb{R}^n$ and $\omega \in \mathbb{R}^n$ by
\[
d^2 \varphi (\bar{x},v)(\omega):= \liminf_{\tau \downarrow0, u \to \omega}  \Delta_{\tau}^2 \varphi(\bar{x},v)(u)
\]
Then $\varphi$ is said to be twice epi-differentiable at $\bar{x}$ for $v$ if for every $\omega \in \mathbb{R}^n$ and every choice of $\tau_k \downarrow 0$ there exists a sequnce $\omega^k \to \omega$ such that 
\[
\frac{\varphi(\bar{x}+\tau_k \omega^k)-\varphi(\bar{x})-\tau_k \langle v,\omega^k \rangle}{\frac{1}{2} \tau_k^2}\to d^2 \varphi(\bar{x},v)(\omega)\ as\ k\to \infty.
\]

Twice epi-differentiability is a fundamental notion in second-order variational analysis and has numerous applications in optimization. We refer the reader to the aforementioned monograph by Rockafellar and Wets, as well as to the recent works~\cite{mohammadi2021parabolic,mohammadi2022variational,mohammadi2020twice}, which develop a symmetric approach to characterizing and verifying epi-differentiability through parabolic regularity.

 We recall the definition of the tilt-stbility for extended-real-valued functions.
 \begin{definition}
    Let $f:\mathbb{R}^n \to \mathbb{\overline{R}}$ and $\bar{x}\in \text{dom}\,f$; then $\bar{x}$ is a \textit{tilt-stable local minimizer} of the function $f$ if there exists $\gamma>0$ such that the argminimum mapping
\begin{equation}
    M_{\gamma}: v \mapsto \text{argmin} \{f(x)-\langle v,x \rangle \ |\ x\in \mathbb{B}_{\gamma}(\bar{x})\}
\end{equation}
is single-valued and Lipshitz continous on a neighborhood of $\bar{v}=0$ with $M_{\gamma}(\bar{v})=\{\bar{x}\}$. A \textit{modulus of tilt stability} of $f$ at $\bar{x}$ is defined as a Lipschitz constant for the mapping $M_\gamma$ in a neighborhood of the origin.
 
 \end{definition}

The following proposition provides a characterization of tilt-stable minimizers; see Theorem~3.2 in~\cite{MordukhovichNghia2015} for its proof.

 \begin{proposition}[Tilt stability via the second-order growth condition]
\label{prop:tilt_stability_second_order_growth}
Let $\varphi : \mathbb{R}^n \to \mathbb{R}$ be continuously prox-regular at $\bar{x}$ for $\bar{v} = 0$. The following are equivalent:
\begin{enumerate}[(i)]
    \item The point $\bar{x}$ is a tilt-stable minimizer of $\varphi$ with modulus $\kappa > 0$.
    
    \item There are neighborhoods $U$ of $\bar{x}$ and $V$ of $\bar{v}$ such that the inverse mapping 
    $(\partial \varphi)^{-1}$ admits a single-valued localization 
    $\gamma : V \to U$ around $(\bar{v}, \bar{x})$, and for any pair 
    $(v,u) \in \mathrm{gph}\,\gamma = \mathrm{gph}((\partial \varphi)^{-1}) \cap (V \times U)$, we have the uniform second-order growth condition
    \begin{equation}\label{eq:second_ordr_condition_of_the_tilt_stable_minimizer}
    \varphi(x) \geq \varphi(u) + \langle v, x - u \rangle + \frac{1}{2\kappa}\|x - u\|^2 \quad \text{whenever } x \in U.
    \tag{2.10}
    \end{equation}
\end{enumerate}
\end{proposition}

Next, we recall some classical notions from convex and variational analysis; see, e.g.,~\cite{rockafellar1998variational}. Let $\varphi:\mathbb{R}^n\to\overline{\mathbb{R}}$ be a proper, l.s.c, extended-real-valued function, and let $\gamma>0$. The \textit{Moreau envelope} of $\varphi$ with parameter $\gamma$ is defined by
\begin{equation}\label{eq:Moreau_Envelope}
    e_{\gamma}\varphi(x)
    :=
    \inf_{y\in\mathbb{R}^n}
    \left\{
        \varphi(y)+\frac{1}{2\gamma}\|y-x\|^2
    \right\},
\end{equation}
and the \textit{proximal mapping} is defined by
\begin{equation}\label{eq:Proximal_mapping}
    \operatorname{Prox}_{\gamma\varphi}(x)
    :=
    \operatorname*{argmin}_{y\in\mathbb{R}^n}
    \left\{
        \varphi(y)+\frac{1}{2\gamma}\|y-x\|^2
    \right\}.
\end{equation}
When $\gamma=1$, we simplify the notation in
\eqref{eq:Moreau_Envelope} and \eqref{eq:Proximal_mapping} to
$e_{\varphi}(x)$ and $\operatorname{Prox}_{\varphi}(x)$, respectively.

In the last part of this section, we recall some general notions and notation that will be used throughout the paper. For any two nonempty sets $A,B\subset\mathbb{R}^n$, their \textit{Hausdorff distance} is defined by
\begin{equation}\label{eq:Hausdorff_distance}
    d_H(A,B)
    :=
    \max\left\{
        \sup_{a\in A}\inf_{b\in B}\|a-b\|,
        \sup_{b\in B}\inf_{a\in A}\|b-a\|
    \right\}.
\end{equation}
For a set $\Omega\subset\mathbb{R}^n$, its \textit{indicator function} $\delta_{\Omega}$ is defined as
\[
    \delta_{\Omega}(x)
    :=
    \begin{cases}
        0, & x\in\Omega,\\
        +\infty, & x\notin\Omega.
    \end{cases}
\]
For a nonempty set $\Omega\subset\mathbb{R}^n$, the distance from $x\in\mathbb{R}^n$ to $\Omega$ is denoted by
\[
    \operatorname{dist}(x;\Omega)
    :=
    \inf_{y\in\Omega}\|x-y\|.
\]
Finally, for a symmetric matrix $H\in\mathbb{R}^{n\times n}$, we denote its smallest and largest eigenvalues by $\lambda_{\min}(H)$ and $\lambda_{\max}(H)$, respectively.

\section{Nonmonotone Coderivative-Based Newton Method}\label{sec:nonmonotone_coderivative_algorithm}
In this section, we study the unconstrained optimization problem~\eqref{eq:unconstrained_main_problem} with a cost function
$f : \mathbb{R}^n \to \mathbb{R}$ of class $C^{1,1}$. This class of optimization problems is of fundamental importance due to its wide range of applications, particularly in machine learning. Representative examples include $\ell_2$-loss support vector regression \cite{ho2012large}, subproblems arising in augmented Lagrangian methods \cite{rockafellar1974augmented,rockafellar2023augmented,hang2022augmented}, as well as other practical applications \cite{hiriart1984generalized}. A coderivative-based extension of the classical Newton method for the local solution of problem~(\ref{eq:unconstrained_main_problem}) was first proposed and analyzed in~\cite{mordukhovich2021generalized}, under the fundamental assumption that the point $\bar{x}$ is a tilt-stable local minimizer of (\ref{eq:unconstrained_main_problem}). To achieve global convergence of the coderivative-based Newton method, the authors of \cite{khanh2024globally,chao2024coderivative} proposed globalization strategies based on backtracking line search techniques for solving \eqref{eq:unconstrained_main_problem}. Building upon these developments, we propose a nonmonotone coderivative-based Newton method that combines three key ingredients: a coderivative-based Newton direction, an adaptive regularization strategy for handling positive semidefinite generalized Hessians, and a Hybrid Adaptive Nonmonotone (HAN) line search scheme for globalization. We first introduce the HAN line search, which serves as the globalization mechanism for the proposed algorithm.

\subsection{Hybrid Adaptive Nonmonotone (HAN) line search}\label{sec:HAN}

Classical globalization strategies for Newton-type methods are commonly based on either Armijo or Wolfe line search techniques. While the Armijo line search is computationally inexpensive, it may accept excessively small step sizes, resulting in slow practical convergence. In contrast, the Wolfe line search often produces larger step sizes but typically requires additional gradient evaluation, leading to a higher computational cost. These complementary strengths and weaknesses motivate the development of a new line search strategy that combines the advantages of both approaches.

\begin{algorithm} \caption{Hybrid Adaptive Nonmonotone (HAN) line search} \label{alg:HAN_linesearch} \begin{algorithmic}[1] \State \textbf{Input:} Search direction $d\in\mathbb{R}^n$, switching thresholds $0 \le C_2 \le C_1 \le1$, memory parameters $M_1,M_2\ge0$, reduction factor $0<\beta<1$, expansion factor $\gamma>1$, and maximum expansion iterations $S\in\mathbb{N}$. \State Set $\alpha\gets1$. \State \textbf{Step 1:} Apply the Armijo line search with reduction factor $\beta$. If an acceptable step size satisfying $\alpha>C_1$ is obtained, terminate; otherwise, proceed to Step 2. \State \textbf{Step 2:} Apply the nonmonotone Armijo line search with memory size $M_1$ and reduction factor $\beta$. If an acceptable step size satisfying $\alpha>C_2$ is obtained, terminate; otherwise, proceed to Step 3. \State \textbf{Step 3:} Apply the nonmonotone Wolfe line search with memory size $M_2$, reduction factor $\beta$, expansion factor $\gamma$, and maximum expansion iterations $S$. \State \Return $\alpha$. \end{algorithmic} \end{algorithm}

To this end, we propose a \emph{Hybrid Adaptive Nonmonotone} (HAN) line search. The HAN strategy begins with the classical Armijo line search because of its low computational cost. If the step size falls below a prescribed threshold, the algorithm switches to a nonmonotone Armijo line search, which remains computationally inexpensive while relaxing the monotonicity requirement to promote the acceptance of larger step sizes. If the step size decreases further, the algorithm transitions to a nonmonotone Wolfe line search. Although this stage incurs a higher computational cost due to the additional curvature condition, it is more effective at preventing excessively small step sizes and can often recover larger, more productive steps. Consequently, the HAN line search adaptively balances computational efficiency with the acceptance of larger step sizes by progressively employing more sophisticated line search strategies only when necessary.

Algorithm~\ref{alg:HAN_linesearch} presents the general framework of the proposed HAN line search. The parameters $C_1$ and $C_2$ ($C_1 \ge C_2$) denote the first and second switching thresholds, respectively. The vector $d$ represents the Newton search direction. The parameter $M_1$ is the memory size used by the nonmonotone Armijo line search, which is activated when the step size $\alpha$ falls below $C_1$, whereas $M_2$ is the memory size used by the nonmonotone Wolfe line search, which is activated when $\alpha$ falls below $C_2$. The parameter $\beta\in(0,1)$ is the backtracking reduction factor, $\gamma>1$ is the step-size expansion factor employed during the Wolfe line search, and $S$ denotes the maximum number of expansion iterations permitted during the Wolfe line search.

Furthermore, the proposed HAN line search provides a unified framework that encompasses the Armijo, nonmonotone Armijo, and Wolfe line search strategies as special cases under appropriate choices of the algorithmic parameters. In particular, when $C_1=C_2=0$, the HAN line search becomes the classical Armijo line search. When $C_1=1$ and $C_2=0$, it reduces to the nonmonotone Armijo line search. When $C_1=C_2=1$ and $M_2=0$, the HAN line search reduces to the Wolfe line search. Moreover, by selecting different combinations of the switching thresholds and the remaining algorithmic parameters, various other hybrid line search strategies can be obtained within the proposed framework. A detailed description of the HAN line search is provided in Appendix~\ref{appendix:HAN}.

\subsection{Algorithm Description}

In this section, we propose the Nonmonotone Coderivative-based Newton (NCN) method, which generalizes the classical regularized Newton method. For the unconstrained optimization problem \eqref{eq:unconstrained_main_problem}, the search direction at each iteration of the NCN algorithm is obtained by solving the following inclusion:
\begin{equation} \label{eq:direction_inclusion_nonmonotone}
- \nabla f(x) \in D^\ast \nabla f(x)(d) + \mu R d,
\end{equation}
where $f$ is a $C^{1,1}$ function, $D^\ast \nabla f(x)$ denotes the generalized Hessian, or coderivative of the gradient mapping, $\mu>0$ is the regularization parameter, and $R$ is a positive definite symmetric regularization matrix.

\begin{algorithm}
  \caption{Nonmonotone Coderivatived-based Newton (NCN) algorithm}
    \label{alg:nomonotone-newton}
\begin{algorithmic}[1]
\State \textbf{Input:} Initial point \(x_0\in\mathbb{R}^n\), parameters \(q>0\) and \(p\in(0,1]\), and the HAN line search parameters \(M_1\), \(M_2\), \(C_1\), \(C_2\), \(\beta\), \(\gamma\), \(c\), \(c^\prime\), and \(S\) defined in Algorithm~\ref{alg:HAN_linesearch} (see the detailed version in Algorithm~\ref{alg:HAN_linesearch2}).

\For{$k = 0,1,2,\dots$}
    \If{$\nabla f(x_k)=0$ \textnormal{ or another stopping criterion is satisfied}}
        \State \textbf{stop}
    \Else
        \State Set $\mu_k$ as a coffecient of $q\|\nabla f(x_k)\|^{p}$ which is explaind in Section~\ref{sec:adaptive_regularization},  and choose $d_k \in \mathbb{R}^n$ such that
        \[
            -\nabla f(x_k) \in D^\ast \nabla f(x_k)(d_k)+\mu_k R_kd_k
        \]
    \EndIf
    \State $\alpha_k = \text{HAN line search(}d_k,C_1,C_2,M_1,M_2,\beta,\gamma, c, c^\prime, S)$ 
    \State Set $x_{k+1} := x_k + \alpha_k d_k$
\EndFor

\end{algorithmic}
\end{algorithm}

Algorithm~\ref{alg:nomonotone-newton} outlines the proposed NCN algorithm. Its well-posedness relies on the solvability of the inclusion \eqref{eq:direction_inclusion_nonmonotone}, which guarantees the existence of a search direction \(d_k\) at each iteration. The following proposition establishes this result in the more general framework allowing for inexact second-order information.

\begin{proposition}\label{prop:well-poseness-nonmonotone-Algorithm}

Let \( \hat{f} : \mathbb{R}^n \to \mathbb{R} \) be a \( C^{1,1} \) function representing either the exact function \(f\) in \eqref{eq:direction_inclusion_nonmonotone} or an approximation thereof. Suppose that \( \nabla f(x) \neq 0 \) and that the coderivative \(D^\ast\nabla\hat{f}(x)\) is positive semidefinite in the sense that
\begin{equation}\label{eq:positive_definitness_exsistance_proposition}
\langle v,u\rangle \ge 0
\quad
\text{for every } u\in\mathbb{R}^n\setminus\{0\}
\text{ and every }
v\in D^\ast\nabla\hat{f}(x)(u).
\end{equation}
Then, for every \(\mu>0\) and every symmetric positive definite matrix \(R\), there exists a nonzero vector \(d\in\mathbb{R}^n\) satisfying
\begin{equation}\label{eq:inclusion_existance_prop}
-\nabla f(x)\in D^\ast\nabla\hat{f}(x)(d)+\mu Rd.
\end{equation}
Moreover, \(d\) is a descent direction for \(f\) at \(x\), namely,
\[
\langle \nabla f(x),d\rangle<0.
\]

\end{proposition}

\begin{proof}
Define the auxiliary function
\[
g(x):=\hat{f}(x)+\frac{\mu}{2}\langle x,Rx\rangle .
\]
Since the quadratic term is \(C^2\), the coderivative sum rule (Theorem 3.9 in \cite{mordukhovich2018variational}) yields
\[
D^\ast\nabla g(x)(d)
=
D^\ast\nabla\hat{f}(x)(d)+\mu Rd.
\]
Hence, it suffices to prove the existence of a vector \(d\) satisfying $
-\nabla f(x)\in D^\ast\nabla g(x)(d).$

We first show that \(D^\ast\nabla g(x)\) is positive definite. Let
\(u\neq0\) and \(v\in D^\ast\nabla g(x)(u)\). Then
$
v=v_1+\mu Ru
$
for some \(v_1\in D^\ast\nabla\hat{f}(x)(u)\); therefore,
$
\langle u,v\rangle
=
\langle u,v_1\rangle
+\mu\langle u,Ru\rangle.
$
The first term is nonnegative by assumption
\eqref{eq:positive_definitness_exsistance_proposition}, while the second is
strictly positive because \(R\) is positive definite and \(u\neq0\). Hence,
$
\langle u,v\rangle>0,
$
showing that \(D^\ast\nabla g(x)\) is positive definite.

By Theorem 5.16 of \cite{mordukhovich2018variational}, the positive definiteness of
\(D^\ast\nabla g(x)\) implies that \(\nabla g\) is strongly locally maximally
monotone around \((x,\nabla g(x))\). Consequently, Theorem 5.13 of
\cite{mordukhovich2018variational} ensures that \(\nabla g\) is strongly metric
regular around this point, and the inverse mapping
$
(\nabla g)^{-1}:\mathbb{R}^n\rightrightarrows\mathbb{R}^n
$
admits a single-valued Lipschitz localization
$
\vartheta:V\to U
$
around \((\nabla g(x),x)\). Applying the scalarization formula gives
\[
D^\ast(\nabla g)^{-1}(\nabla g(x),x)(\nabla f(x))
=
D^\ast\vartheta(\nabla g(x))(\nabla f(x))
=
\partial\langle\nabla f(x),\vartheta\rangle(\nabla g(x)).
\]
Since \(\vartheta\) is locally Lipschitz, Theorem 1.22 of
\cite{mordukhovich2018variational} guarantees that $\partial \langle \nabla f(x),\vartheta \rangle(\nabla g(x))\neq \emptyset$.
Hence, we may choose
\[
-d\in
D^\ast(\nabla g)^{-1}(\nabla g(x),x)(\nabla f(x)).
\]
Using the definition of the limiting coderivative, it is straightforward to obtain
\[
-\nabla f(x)\in D^\ast\nabla g(x)(d),
\]
which, together with the coderivative sum rule, gives
\[
-\nabla f(x)
\in
D^\ast\nabla\hat{f}(x)(d)+\mu Rd.
\]

It remains to prove that \(d\neq0\). Suppose, to the contrary, that \(d=0\).
Since \(\nabla g\) is locally Lipschitz, Theorem 3.3 of
\cite{mordukhovich2018variational} implies that
\[
D^\ast\nabla g(x)(0)=\{0\}.
\]
Therefore,
\[
-\nabla f(x)\in D^\ast\nabla g(x)(0)=\{0\},
\]
which contradicts the assumption that \(\nabla f(x)\neq0\). Thus,
\(d\neq0\). Finally, we have
\[
-\nabla f(x)-\mu Rd
\in
D^\ast\nabla\hat{f}(x)(d).
\]
By the positive semidefiniteness of \(D^\ast\nabla\hat{f}(x)\),
\[
\langle d, -\nabla f(x)-\mu Rd\rangle\ge0,
\]
which implies
$
\langle-\nabla f(x),d\rangle
\ge
\mu\langle d,Rd\rangle.
$
Since \(R\) is positive definite, $\mu>0$, and \(d\neq0\),
$
\mu\langle d,Rd\rangle>0.
$
Hence,
$
\langle-\nabla f(x),d\rangle>0,
$
or equivalently,
\[
\langle\nabla f(x),d\rangle<0.
\]
Thus, \(d\) is a descent direction for \(f\).

\end{proof}

Proposition~\ref{prop:well-poseness-nonmonotone-Algorithm} establishes the well-posedness of the direction-finding step in Algorithm~\ref{alg:nomonotone-newton}. In particular, under the positive semidefiniteness assumption on the coderivative mapping, the inclusion \eqref{eq:direction_inclusion_nonmonotone} admits a nonzero solution at every nonstationary iterate. Moreover, the resulting direction is a descent direction for the objective function. Consequently, the search direction computed by the NCN algorithm is always well defined, allowing the HAN line search to be applied to determine a suitable step length.

It is worth noting that the formulation in terms of the auxiliary function ($\hat{f}$) provides a unified framework that accommodates both exact and approximate second-order information. In the proposed NCN algorithm, the exact second-order information is employed ($\hat{f}=f$). Nevertheless, the framework naturally extends to the case where ($\hat{f}$) is an approximation of (f), thereby yielding inexact variants while preserving the well-posedness of the search-direction computation under the assumptions of Proposition~\ref{prop:well-poseness-nonmonotone-Algorithm}.

\subsection{Adaptive Regularization Strategies}\label{sec:adaptive_regularization}

In this section, we present two adaptive regularization strategies employed by the NCN algorithm. In both strategies, the regularization parameter is updated dynamically based on information from the preceding iteration. The first strategy is motivated by the trust-region globalization framework, whereas the second exploits the step size returned by the line search procedure.

\subsubsection{Trust-region-based adaptive strategy}

This approach predicts the value of the objective function using the following local quadratic model of $f$ at $x_k$:

\begin{equation}\label{eq:predict_equation}
m_k(d)
:=
f(x_k)
+\langle \nabla f(x_k),d\rangle
+\frac{1}{2}\langle d,\omega_k\rangle
+\frac{\mu_k}{2}\langle d,Rd\rangle,
\end{equation}
where
$\omega_k\in D^\ast\nabla f(x_k)(d_k)$ and $R$ is a positive definite symmetric regularization matrix. Under the assumed positive semidefiniteness of the coderivative term, the inclusion \eqref{eq:direction_inclusion_nonmonotone} represents the first-order optimality condition associated with the local quadratic model \eqref{eq:predict_equation}. Hence, the resulting search direction $d_k$ is a minimizer of the regularized local model.

To evaluate the quality of the local quadratic model, we compare the actual reduction in the objective function with the reduction predicted by the model. Specifically, we define the ratio
\begin{equation}\label{eq:rho_definition}
\rho_k
:=
\frac{f(x_k)-f(x_k+d_k)}
{m_k(0)-m_k(d_k)}.
\end{equation}
Here, the numerator represents the actual reduction in the objective function, while the denominator represents the reduction predicted by the local model. Thus, $\rho_k$ quantifies the accuracy of the model in predicting the actual reduction. A value of $\rho_k$ close to one indicates good agreement between the actual and predicted reductions, whereas a small value of $\rho_k$ indicates that the model overestimates the reduction achieved by the computed step.

Since $d_k$ satisfies the first-order optimality condition associated with \eqref{eq:predict_equation}, there exists $\omega_k \in D^\ast \nabla f(x_k)(d_k)$ such that 
\[
-\nabla f(x_k)=\omega_k+\mu_k R_k d_k
\]
Consequently
\[
m_k(0)-m_k(d_k)
=
-\frac{1}{2}\left\langle\nabla f(x_k),d_k\right\rangle,
\]
and hence the ratio in \eqref{eq:rho_definition} can equivalently be written as
\begin{equation}\label{eq:rho_simplified}
\rho_k
=
\frac{f(x_k)-f(x_k+d_k)}
{-\frac{1}{2}\langle\nabla f(x_k),d_k\rangle}.
\end{equation}
A sufficiently large value of $\rho_k$ indicates good agreement between the actual and predicted reductions. In this case, the regularization parameter can be decreased. On the other hand, a small value of $\rho_k$ indicates a poor agreement between the model prediction and the actual reduction, motivating an increase in the regularization parameter. Accordingly, we update the regularization parameter as follows:
\begin{equation}\label{eq:Adaptive_reg_rule_TR}
\mu_{k+1}
=
\begin{cases}
\eta_1 q\|\nabla f(x_k)\|^p,
& \rho_k\ge T_2,\\[2mm]
q\|\nabla f(x_k)\|^p,
& T_1<\rho_k<T_2,\\[2mm]
\eta_2 q\|\nabla f(x_k)\|^p,
& \rho_k\le T_1,
\end{cases}
\end{equation}
where $q>0$, $p> 0$, and $0<\eta_1<1<\eta_2$, with $0<T_1<T_2$.

\subsubsection{line search-based adaptive strategy}

The second strategy uses the step size generated by the line search procedure as an indicator of the quality of the computed search direction. In particular, if the line search accepts the full step, i.e., $\alpha_k=1$, the direction is regarded as sufficiently reliable, and the regularization parameter is reduced. In contrast, a substantially shortened step indicates that the computed direction may be less suitable, in which case a larger regularization parameter is employed.

More precisely, the regularization parameter is updated according to
\begin{equation}\label{eq:Adaptive_reg_rule_linesearch}
\mu_{k+1}
=
\begin{cases}
\eta_1 q\|\nabla f(x_k)\|^p,
& \alpha_k=1,\\[2mm]
q\|\nabla f(x_k)\|^p,
& T_1<\alpha_k<1,\\[2mm]
\eta_2 q\|\nabla f(x_k)\|^p,
& \alpha_k\le T_1,
\end{cases}
\end{equation}
where $0<T_1<1$ and $0<\eta_1<1<\eta_2$.

For both strategies, the initial regularization parameter is selected as
\begin{equation}
\mu_1=q\|\nabla f(x_1)\|^p,
\end{equation}
and, for $k\geq1$, $\mu_{k+1}$ is subsequently determined by either the trust-region-based rule \eqref{eq:Adaptive_reg_rule_TR} or the line search-based rule \eqref{eq:Adaptive_reg_rule_linesearch}.

\section{Convergence Analysis }\label{sec:convergence_analysis}

In this section, we analyze the convergence properties of the proposed NCN algorithm, establishing both global convergence and local superlinear convergence. Although the NCN algorithm is formulated using exact second-order information, the analysis is carried out within a more general framework that also accommodates inexact second-order information. Consequently, the convergence results apply not only to the NCN algorithm but also to a broader class of Newton-type methods.

The following proposition establishes the global convergence of the coderivative-based Newton-type methods employing exact or inexact second-order information. The result is global in the sense that convergence is guaranteed for arbitrary initial iterates, without requiring the starting point to lie in a neighborhood of a minimizer. Furthermore, the analysis is sufficiently general to encompass semistochastic generalized Newton methods and other Newton-type algorithms based on approximate second-order information.

\vspace{0.1cm}
\begin{proposition}[\textbf{Global Convergence}]
\label{prop:global_convergent_nonmonotone}
In the setting of Proposition~\ref{prop:well-poseness-nonmonotone-Algorithm}, let \(x_0\in\mathbb{R}^n\) be the initial point, and define the level set
\[
\Omega:=\{x\in\mathbb{R}^n\mid f(x)\le f(x_0)\}.
\]
Assume that \(\Omega\) is bounded and that \(D^*(\nabla\hat{f})(x)\) is positive semidefinite on \(\Omega\) in the sense of \eqref{eq:positive_definitness_exsistance_proposition}. Then the sequence \(\{x_k\}\) generated by \eqref{eq:inclusion_existance_prop} satisfies
\[
\lim_{k\to\infty}\|\nabla f(x_k)\|=0.
\]
\end{proposition}

\begin{proof}
In the first part of the proof, we show that the direction \( d_k \), defined by \eqref{eq:inclusion_existance_prop}, satisfies conditions ($\mathcal{B}$\ref{condition:decent_direction}), ($\mathcal{B}$\ref{condition:Angle_conditon}), and ($\mathcal{B}$\ref{condition:bounded_direction}) of Lemma~\ref{lemma:HAN_convergence}. Condition ($\mathcal{B}$\ref{condition:decent_direction}) follows directly from Proposition~\ref{prop:well-poseness-nonmonotone-Algorithm}. Since 
\[
-\nabla f(x_k) \in D^\ast \nabla f(x_k)(d_k) + \mu_k R_k d_k,
\]
we can rewrite this inclusion as $
-\nabla f(x_k) - \mu_k R_k d_k \in D^\ast \nabla f(x_k)(d_k).
$
Using the positive semidefiniteness of \(\nabla f(x_k)\), it follows that $
\langle -\nabla f(x_k) - \mu_k R_k d_k,\, d_k \rangle \ge 0,
$
which yields
\begin{equation}\label{eq:global_convergece_nonmonotone_algorithm_A1}
\langle -\nabla f(x_k), d_k \rangle \ge \mu_k \langle R_k d_k,d_k \rangle \ge \mu_k \lambda_{\min(R_k)}\|d_k\|^2=t \|\nabla f(x_k)\| ^p\|d_k\|^2,
\end{equation}
where $t=\lambda_{\min(R_k)} \eta q$, and $\eta$ could be 1, $\eta_1$, or $\eta_2$ based on adaptive regularization in (\ref{eq:Adaptive_reg_rule_TR}) and (\ref{eq:Adaptive_reg_rule_linesearch}). This means that $d_k$ satisfies condition ($\mathcal{B}$\ref{condition:Angle_conditon}) with $\gamma=p$ and $\zeta=2$.
Next, from \eqref{eq:global_convergece_nonmonotone_algorithm_A1}, we deduce
\begin{equation}\label{eq:global_convergece_nonmonotone_algorithm_A3}
\|\nabla f(x_k)\| \, \|d_k\| 
\ge \langle -\nabla f(x_k), d_k \rangle 
\ge \mu_k \lambda_{\min(R_k)}\|d_k\|^2,
\end{equation}
which implies $
\mu_k \lambda_{\min(R_k)} \|d_k\| \le \|\nabla f(x_k)\|.$ Thus, under assumption $\|\nabla f(x_k)\|>\epsilon$ for some positive value $\epsilon$, we obtain $\|d_k\| \le \frac{1}{\mu_k \lambda_{\min(R_k)}} \|\nabla f(x_k)\| \le \frac{1}{t\epsilon^p}\|\nabla f(x_k)\|$. This indicate that condition ($\mathcal{B}$\ref{condition:bounded_direction}) is satisfied by $d_k$.

Finally, it is clear that the condition ($\mathcal{B}$\ref{condition:Nonmonotone_linesearch}) of Lemma \ref{lemma:HAN_convergence} is satisfied by Algorithm \ref{alg:nomonotone-newton} since $\alpha_k$ is chosen based on the HAN line search. Therefore, it is straightforward to prove $
\lim_{k \to \infty} \|\nabla f(x_k)\| = 0
$ employing Lemma \ref{lemma:HAN_convergence}.
    
\end{proof}

The significance of Proposition~\ref{prop:global_convergent_nonmonotone} is twofold. First, it establishes global convergence of the proposed NCN algorithm from arbitrary initial points under standard bounded level-set and positive-semidefiniteness assumptions. Second, the result is independent of the particular choice of second-order information. In particular, the convergence analysis remains valid when the exact generalized Hessian is replaced by any approximation whose coderivative satisfies the positive semidefiniteness condition \eqref{eq:positive_definitness_exsistance_proposition}. 

A principal advantage of second-order methods over first-order methods is their superior local convergence behavior. While first-order methods typically achieve at most linear convergence, Newton-type methods can attain superlinear or even quadratic convergence once the iterates enter a neighborhood of a solution, provided suitable regularity conditions are satisfied. This fast local convergence is one of the main reasons for the widespread use of second-order methods.

The following proposition provides a unified framework for analyzing the local convergence of the proposed NCN algorithm. The framework is designed to handle both exact and inexact second-order information, covering a broad range of generalized Newton-type methods. In particular, the analysis applies to deterministic methods as well as semistochastic and inexact variants, providing a common basis for studying their local convergence.

\begin{proposition}[\textbf{\color{black}Superlinear local convergence rate}]\label{prop:superlinear_convergent_rate}  In the setting of Propositon~\ref{prop:well-poseness-nonmonotone-Algorithm}, 
\deleted{ Let \( f : \mathbb{R}^n \to \mathbb{R} \) be of \( C^{1,1} \) function and} let \( \bar{x} \) be an accumulation point of a sequence \( \{x_k\} \). Assume that \( \nabla f \) is semismooth\(^*\) at $\bar{x}$ and metrically regular in a neighborhood of \( \bar{x} \). In addition, suppose that the approximated generalized hessian $D^\ast \nabla \hat{f}(x_k)$ is positive semidefinite on $\Omega$ and satisfy the following condition
\begin{equation}\label{eq:superlinear_inexact_condition}
\sup_{\|v\|=1} d_H\Big(D^\ast \nabla \hat{f}(x_k)(v),D^\ast \nabla f(x_k)(v)\Big )=o(\|x_k-\bar{x}\|).
\end{equation}
Then \( \bar{x} \) is a tilt-stable local minimizer of \( f \). Furthermore, if at least one of the following conditions holds:
\begin{enumerate}
\item[i)] \( \nabla f \) is directionally differentiable at \( \bar{x} \);
\item[ii)] \( c \in \left(0, \frac{1}{4l\kappa}\right) \) and $c^\prime \in (1-\frac{\kappa}{l},1)$, where \( \kappa > 0 \) and \( l > 0 \) are the moduli of metric regularity and Lipschitz continuity of \( \nabla f \) around \( \bar{x} \),
\end{enumerate}
then the sequence \( \{x_k\} \) converges to \( \bar{x} \) with Q-superlinear rate. 

\end{proposition}

\begin{proof} For clarity, we divide the proof into several claims and combine them to establish the desired result.\\
\textbf{Claim 1.} $D^\ast\nabla f(\bar{x})$ \textit{is positive semidefinite}. Under condition~(\ref{eq:superlinear_inexact_condition}) and the positive homogeneity of the coderivative, for every $u\in D^\ast\nabla f(\bar{x})(v)$, there exists a sequence $\{u_k\}$ with
$u_k\in D^\ast\nabla\hat{f}(x_k)(v)$ such that
\begin{equation}\label{eq:limit_claim_1_local_superlinear}
\lim_{k\to\infty}\langle u_k,v\rangle
=\langle u,v\rangle,
\qquad v\in\mathbb{R}^n.
\end{equation}
Since $D^\ast\nabla\hat{f}(x_k)$ is positive semidefinite, we have $\langle u_k,v\rangle \geq 0.$ Therefore, by~(\ref{eq:limit_claim_1_local_superlinear}), it follows that $\langle u,v\rangle \geq 0,$ which implies that $D^\ast\nabla f(\bar{x})$ is positive semidefinite.
\\
\textbf{Claim 2.}  \textit{$\bar{x}$ is a tilt-stable local minimizer.} By Proposition~\ref{prop:global_convergent_nonmonotone}, $\bar{x}$ is a stationary point and $\bar{x}\in\Omega$. In addition, Claim 1 shows that the coderivative $D^\ast\nabla f(\bar{x})$ is positive semidefinite. Together with the metric regularity of $\nabla f$ around $\bar{x}$ with modulus $\kappa$, these properties allow us to apply Theorem~4.13 of \cite{drusvyatskiy2014second}, which implies that $\bar{x}$ is a tilt-stable local minimizer of $f$ with modulus $\kappa$.\\
\textbf{Claim 3.} \textit{The following equality holds}:
 \begin{equation}\label{eq:nonmonotone_superlinear_relationship}
    \|x_k + d_k - \bar{x}\| = o(\|x_k - \bar{x}\|) \quad \text{as } k \to \infty.
\end{equation}  
Based on Lemma 5.6 in \cite{duy2023generalized}, we have 
\begin{equation}
    D^\ast \nabla f(x_k)(d_k) \subset D^\ast \nabla f(x_k)(x_k+d_k-\bar{x})+D^\ast \nabla f(x_k)(\bar{x}-x_k).
\end{equation}
Using \eqref{eq:inclusion_existance_prop}, it follows that
\begin{equation}
 -\nabla f(x_k) - \mu_k R_k d_k \in D^\ast \nabla \hat{f}(x_k)(d_k).   
\end{equation}
Therefore, using (\ref{eq:superlinear_inexact_condition}), there exist $\epsilon_k\in \mathbb{R}^n$ such that $\|\epsilon_k\|=o(\|x_k-\bar{x}\|)$ and 
\begin{equation}\label{eq:superlinear_convergence_nonmonotone_algoithm_equation1}
     -\nabla f(x_k) - \mu_k R_k d_k +\epsilon_k\in D^\ast \nabla f(x_k)(d_k).  
\end{equation}
Also, there exists $v_k \in D^\ast \nabla f(x_k)(\bar{x} - x_k)$ such that
\begin{equation}\label{eq:superlinear_convergence_nonmonotone_algoithm_equation2}
    -\nabla f(x_k) - \mu_k R_k d_k - v_k  +\epsilon_k\in D^\ast \nabla f(x_k)(x_k + d_k - \bar{x}).
\end{equation}
Since $\bar{x}$ is a tilt-stable minimizer, its characterization via the combined second-order subdifferential—see Theorem 3.5 in \cite{MordukhovichNghia2015} and Proposition 4.6 in \cite{chieu2017second}—yields the existence of $\delta > 0$ such that, for all $x \in \mathbb{B}_{\delta}(\bar{x})$, the coderivative $D^\ast \nabla f(x)$ is positive definite in the following sense:
\begin{equation}\label{eq:titl-stable-second-order-characterization}
    \langle z, \omega \rangle \ge \frac{1}{\kappa} \| \omega \|^2 
    \quad \text{for all } z \in D^\ast \nabla f(x)(\omega),\; x \in \mathbb{B}_\delta(\bar{x}),\; \text{and } \omega \in \mathbb{R}^n.
\end{equation}

Therefore, for all sufficiently large $k \in \mathbb{N}$ such that $x_k \in \mathbb{B}_{\delta}(\bar{x})$, we can apply \eqref{eq:titl-stable-second-order-characterization} to (\ref{eq:superlinear_convergence_nonmonotone_algoithm_equation1}) and (\ref{eq:superlinear_convergence_nonmonotone_algoithm_equation2}), which yields:

\begin{equation}\label{eq:super_linear_convergence_nonmonotone_direction_estimate}
    \begin{split}
        & \langle- \nabla f(x_k) - \mu _kR_k d_k+\epsilon_k,d_k \rangle \ge \frac{1}{\kappa}\|d_k\|^2\\ &\Rightarrow \langle \nabla f(x_k), -d_k \rangle - \mu _k \|d_k\|_{R_k}^2 +\langle \epsilon_k,d_k\rangle \ge \frac{1}{\kappa}\|d_k\|^2\\
        & \Rightarrow \langle \nabla f(x_k), -d_k \rangle +\langle\epsilon_k,d_k \rangle \ge \kappa^{-1} \|d_k\|^2+\mu_k\|d_k\|_{R_k}^2\\ 
        & \ge (\kappa^{-1}+\mu_k \lambda_{\min(R_k)})\|d_k\|^2 \ge \kappa^{-1}\|d_k\|^2,
    \end{split}
\end{equation}
and
\begin{equation}\label{eq:super_linear_convergence_nonmonotone_error_estimate}
  \langle- \nabla f(x_k)- \mu_k R_k d_k - v_k+\epsilon_k, x_k+d_k-\bar{x} \rangle \ge \frac{1}{\kappa} \|x_k+d_k-\bar{x}\|^2.
\end{equation}
Using the Cauchy–Schwarz inequality, $\nabla f(\bar{x}) = 0$, Lipschitz continuity of $\nabla f$, and (\ref{eq:super_linear_convergence_nonmonotone_direction_estimate}), it follows that:
\begin{equation}\label{eq:super_linear_convergence_nonmonotone_direction_estimate2}
\begin{split}
    &   \frac{1}{\kappa}\|d_k\|^2 \le \langle \nabla f(x_k), -d_k \rangle +\langle \epsilon_k,d_k \rangle \le \|d_k\| \|\nabla f(x_k)\|+\|\epsilon_k\|\|d_k\| \Rightarrow \\
    &\|d_k\| \le \kappa \| \nabla f(x_k) \| +\kappa \|\epsilon_k\|= \kappa \|\nabla f(x_k)-\nabla f(\bar{x})\| +\kappa \|\epsilon_k\|\le \kappa l \|x_k-\bar{x}\|+\kappa \|\epsilon_k\|.\\
\end{split}
\end{equation}
Also, employing (\ref{eq:super_linear_convergence_nonmonotone_error_estimate}) and Cauchy-Schwarz inequality, it is obtained:

\begin{equation}\label{eq:super_linear_convergence_nonmonotone_error_estimate2}
\begin{split}
    & \frac{1}{\kappa}\|x_k+d_k-\bar{x}\|^2 \le \langle - \nabla f(x_k)-\mu_kR_kd_k-v_k+\epsilon_k,x_k+d_k-\bar{x}\rangle \\
    & \le \|-\nabla f(x_k)-\mu_kR_kd_k-v_k+\epsilon_k\|\|x_k+d_k-\bar{x}\| \Rightarrow \\
    & \|x_k+d_k-\bar{x}\| \le \kappa \|\nabla f(x_k)+v_k+\mu_k R_kd_k-\epsilon_k\| \\
    &\le \kappa(\|\nabla f(x_k)+v_k\|+\mu_k\|R_kd_k\|+\|\epsilon_k\|) \\
    &\le \kappa(\|\nabla f(x_k)+v_k\|+\mu_k \lambda_{\max(R_k)}\|d_k\|+\|\epsilon_k\|).
\end{split}
\end{equation}
The following inequality follows from the Lipschitz continuity of $\nabla f$, together with $\nabla f(\bar{x}) = 0$, and $\mu_k = \eta q\|\nabla f(x_k)\|^p$ (where $\eta$ is a coeffiecnt determined by (\ref{eq:Adaptive_reg_rule_TR}) or (\ref{eq:Adaptive_reg_rule_linesearch})):
\begin{equation}\label{eq:eq:super_linear_convergence_nonmonotone_mu_k_estimate}
    \mu_k=\eta q \|\nabla f(x_k)\|^p=\eta q \|\nabla f(x_k)-\nabla f(\bar{x})\|^p\le \eta ql \|x_k-\bar{x}\|^p.
\end{equation}
Using (\ref{eq:super_linear_convergence_nonmonotone_direction_estimate2}), (\ref{eq:super_linear_convergence_nonmonotone_error_estimate2}), and (\ref{eq:eq:super_linear_convergence_nonmonotone_mu_k_estimate}):
\begin{equation}\label{eq:super_linear_convergence_nonmonotone_error_estimate3}
    \|x_k+d_k-\bar{x}\| \le \kappa \|\nabla f(x_k)+v_k\|+\kappa^2 l \lambda_{\max(R_k)} \eta q \Big( l\|x_k-\bar{x}\|^{1+p}+ \|x_k-\bar{x}\|^p\|\epsilon_k\|\Big)+\kappa \|\epsilon_k\|.
\end{equation}
By the semismoothness* of $\nabla f$ at $\bar{x}$ and the fact that $\nabla f(\bar{x}) = 0$, it holds that
\begin{equation}\label{eq:super_linear_convergence_nonmonotone_final_estimate}
\|\nabla f(x_k)+v_k\| = \|\nabla f(x_k) - \nabla f(\bar{x}) + v_k\| = o(\|x_k - \bar{x}\|).
\end{equation}
Consequently, by (\ref{eq:super_linear_convergence_nonmonotone_error_estimate3}) and (\ref{eq:super_linear_convergence_nonmonotone_final_estimate}), it follows that
\begin{equation}\label{eq:superlineae_local_rate_last_equation}
\|x_k + d_k - \bar{x}\| = o(\|x_k - \bar{x}\|).
\end{equation}

In the second part of the proof, we establish that the algorithm eventually accepts the full Newton step at every iteration. More precisely, there exists an index $\bar{k} \in \mathbb{N}$ such that $\alpha_k = 1$ for all $k > \bar{k}$. To this end, we establish the two following claims.\\
\textbf{Claim 4. }\textit{the unit step size is accepted under assumption (i) for sufficiently large $k$.} Since $\nabla f$ is semismooth$^*$, it is also semismooth at $\bar{x}$. By \eqref{eq:angle_estimation_with_direction_norm} and Proposition~\ref{prop:unit_stepsize_semismooth}, condition~\eqref{eq:nonmonotone-condition-HAN} is satisfied. Moreover, Proposition~\ref{prop:super_linear_local_rate_curvuture_condition} guarantees that condition~\eqref{eq:wolf_condition_HAN} also holds. Therefore, the unit step size is eventually accepted by the HAN line search.\\
\textbf{Claim 5. }\textit{the unit step size is accepted under assumption (ii) for sufficiently large $k$.} We begin with the following inequality, obtained from the relation $\|d_k\| = \|x_{k+1} - x_k\| = \|x_{k+1} - \bar{x} + \bar{x} - x_k\|$:
\begin{equation}
    | \frac{\|x_k-\bar{x}\|}{\|x_k-\bar{x}\|}-\frac{\|x_{k+1}-\bar{x}\|}{\|x_k-\bar{x}\|}|\le \frac{\|d_k\|}{\|x_k-\bar{x}\|}\le \frac{\|x_{k+1}-\bar{x}\|}{\|x_k-\bar{x}\|}+\frac{\|x_k-\bar{x}\|}{\|x_k-\bar{x}\|}.
\end{equation}
Consequently,
\begin{equation}\label{eq:super_linear_convergence_nonmonotone_direction_error_limit_1}
  \lim_{k \to \infty} \frac{\|d_k\|}{\|x_k - \bar{x}\|} = 1  
\end{equation}
which indicates $
\lim_{k \to \infty} 
\frac{\|x_{k+1} - x_k\| / \|x_k - \bar{x}\|}{\|d_k\| / \|x_k - \bar{x}\|} = 0
$ as well as
\begin{equation}
 \|x_{k+1} - \bar{x}\| = o(\|d_k\|).   
\end{equation}
Using (\ref{eq:super_linear_convergence_nonmonotone_direction_error_limit_1}), for sufficient large $k$ we have:
\begin{equation}
    |\frac{\|x_k-\bar{x}\|}{\|d_k\|}-1|\le \frac{1}{4l\kappa}.
\end{equation}
 Without loss of generality, consider Proposition~\ref{prop:tilt_stability_second_order_growth} and assume that there exists $\delta>0$ such that $
\mathbb{B}_\delta(\bar{x}) \times \mathbb{B}_{l\delta}(\nabla f(\bar{x})) \subset U \times V$.
Since $x_k \to \bar{x}$, it follows that $(x_k, \nabla f(x_k)) \to (\bar{x}, \nabla f(\bar{x}))$, and hence, for all sufficiently large $k$,
\[
(x_k, \nabla f(x_k)) \in \mathbb{B}_\delta(\bar{x}) \times \mathbb{B}_{l\delta}(\nabla f(\bar{x})).
\]
Therefore, the second-order growth condition applied at $(x_{k+1}, \nabla f(x_{k+1}))$ that yields
\begin{equation}\label{eq:nonmonotone_superlinear_second_order_growth_estimation}
    f(x_k) \ge f(x_{k+1}) + \langle \nabla f(x_{k+1}), x_k - x_{k+1} \rangle + \frac{1}{2\kappa}\|x_k - x_{k+1}\|^2.
\end{equation}
Note that in the HAN line search $F_k:=\max_{j=0,1,\dots,M} f(x_{k-j})$ which indicates $F_k\ge f(x_k)$. Using (\ref{eq:nonmonotone_superlinear_second_order_growth_estimation}) and (\ref{eq:nonmonotone-condition-HAN}), we obtain:  
\begin{equation}
\begin{split}
    &  f(x_{k+1)}-F_k-c \langle\nabla f(x_k),d_k \rangle \le \langle \nabla f(x_{k+1}) ,d_k \rangle-\frac{1}{2\kappa}\|d_k\|^2-c \langle \nabla f(x_k),d_k \rangle \\
    & \le \|\nabla f(x_{k+1})\|.\|d_k\|-\frac{1}{2\kappa} \|d_k\|^2+c \|\nabla f(x_k)\|\|d_k\|\\
    & \le l \|x_{k+1}-\bar{x}\|.\|d_k\|-\frac{1}{2 \kappa} \|d_k\|^2+cl\|x_k-\bar{x}\|\|d_k\|\\
    &=\|d_k\|^2 \Bigg( (\frac{o(\|d_k\|^2)}{\|d_k\|^2})+l(c\frac{\|x_k-\bar{x}\|}{\|d_k\|}-\frac{1}{2l\kappa}) \Bigg)
\end{split} 
\end{equation}
To complete the proof, it is sufficient to show that $c\frac{\|x_k-\bar{x}\|}{\|d_k\|}-\frac{1}{2l\kappa}<0$. To do this we have:
\begin{equation}
    c\frac{\|x_k-\bar{x}\|}{\|d_k\|}-\frac{1}{2l\kappa}<\frac{c}{4l\kappa}+c-\frac{1}{2l\kappa}\le \frac{1}{4l \kappa}+c-\frac{1}{2l\kappa}=c-\frac{1}{4l\kappa}<0.
\end{equation}
As, $k\to \infty$, 
\[
\Bigg( (\frac{o(\|d_k\|^2)}{\|d_k\|^2})+l(c\frac{\|x_k-\bar{x}\|}{\|d_k\|}-\frac{1}{2l\kappa}) \Bigg)\to A<0.
\]
So, there exists $\bar{k}\in \mathbb{N}$ such that for $k>\bar{k}$, we have:
\[
f(x_{k+1)}-F_k-c \langle\nabla f(x_k),d_k \rangle <0,
\]
which means that for $k>\bar{k}$, $\alpha_k=1$ passes condition (\ref{eq:nonmonotone-condition-HAN}). Moreover, Proposition~\ref{prop:super_linear_local_rate_curvuture_condition} ensures that the unit step size $\alpha_k=1$ eventually satisfies condition~\eqref{eq:wolf_condition_HAN} and is consequently accepted by the HAN line search.
\\
\textbf{Claim 6. }\textit{the sequence $\{x_k\}$ converges to $\{\bar{x}\}$ with Q-superlinear rate.} Since Claims 4 and 5 establish that the unit stepsize is accepted for all sufficiently large $k$, combining this result with \eqref{eq:nonmonotone_superlinear_relationship} yields the local superlinear convergence rate of the algorithm.

\end{proof}

In the following corollary, we show that local superlinear convergence can still be achieved when the approximation error in the generalized Hessian vanishes sufficiently fast only along the direction of the algorithmic step.

\begin{corollary}
Under the setting of Proposition~\ref{prop:superlinear_convergent_rate}, suppose that $D^\ast \nabla f(x_k)$ is positive semidefinite. Then, the algorithm retains its local superlinear convergence rate if condition~\eqref{eq:superlinear_inexact_condition} is replaced by the weaker condition
\begin{equation}
d_H\Big(D^\ast \nabla \hat{f}(x_k)(d_k),D^\ast \nabla f(x_k)(d_k)\Big)
=o(\|x_k-\bar{x}\|),
\end{equation}
where $d_k$ is a solution of~\eqref{eq:inclusion_existance_prop}.
\end{corollary}

\begin{proof}
The proof follows the same arguments as those in Proposition~\ref{prop:superlinear_convergent_rate}. In particular, the first claim follows directly from the positive semidefiniteness of $D^\ast \nabla f(x_k)$.

\end{proof}

Proposition \ref{prop:superlinear_convergent_rate} provides a unified framework for establishing local superlinear convergence of coderivative-based Newton methods, even when approximate second-order information is used. The key condition is that the approximation error decreases faster than the distance to the solution. In the semi-stochastic setting, this condition can be ensured by appropriately increasing the number of samples used to approximate the second-order structure. This sampling-based approach to controlling the second-order approximation error has been considered in previous works on semi-stochastic variants of the classical Newton method; see, for example, [12,14]. Together, Propositions~\ref{prop:superlinear_convergent_rate} and \ref{prop:global_convergent_nonmonotone} show that coderivative-based Newton methods can be extended to inexact and semi-stochastic settings while preserving global convergence and local superlinear convergence under suitable error control.

To summarize, we compare our method with existing generalized Newton methods and highlight their main differences. First, Algorithm~\ref{alg:nomonotone-newton} uses the Mordukhovich coderivative of the gradient mapping to construct the generalized Hessian, whereas most existing generalized Newton methods rely on the Clarke generalized Jacobian. Recently, two coderivative-based Newton methods were proposed in \cite{khanh2024globally,chao2024coderivative}. The first employs an Armijo line search to globalize the Newton step, while the second uses a Wolfe line search. Although these globalization strategies can be effective, the former may lead to excessively small step sizes, while the latter may incur high computational costs. To address these limitations, our method employs the HAN line search, which provides a more flexible globalization strategy. The second difference is the regularization strategy. Existing coderivative-based Newton methods use a fixed regularization formula, while our method adjusts the regularization parameter based on the behavior of the algorithm, as described in Section~\ref{sec:adaptive_regularization}. Finally, our convergence analysis is more general than that of existing coderivative-based Newton methods, as it allows the generalized Hessian to be computed either exactly or approximately. This makes the framework applicable to both deterministic methods with exact second-order information and inexact or semistochastic variants based on approximation or sampling.

\section{Application to Convex Composite Optimization Problems}\label{sec:application_lasso}

In this section, we apply the proposed coderivative-based Newton framework to the class of \emph{convex composite optimization problems}. This class encompasses many important optimization models arising in machine learning, statistics, signal processing, and image reconstruction. A convex composite optimization problem consists of minimizing the sum of a smooth convex function and a possibly nonsmooth convex regularization term. Its general form is given by
  
\begin{equation}\label{eq:composite_convex_definition}
    \min_{x \in \mathbb{R}^n} \; \varphi(x) := h(x) + g(x),
\end{equation}
where $h : \mathbb{R}^n \to \mathbb{R}$ is a convex and smooth function, and the regularizer $g : \mathbb{R}^n \to \overline{\mathbb{R}}$ is a convex, extended real-valued one.

The optimization problem in \eqref{eq:composite_convex_definition} is a first-order nonsmooth problem, whereas coderivative-based Newton methods are designed to solve second-order nonsmooth systems. To bridge this gap, we employ the \emph{forward--backward envelope} (FBE), which provides an equivalent reformulation of \eqref{eq:composite_convex_definition}. Specifically, the FBE preserves the set of minimizers of the original problem while endowing it with a second-order nonsmooth structure that is well suited to the application of coderivative-based Newton methods.

\begin{definition}[\textbf{forward-backward envelope}]
Let $\varphi = h + g$ be as defined in (\ref{eq:composite_convex_definition}), and let $\gamma > 0$. The \emph{forward--backward envelope} \textnormal{(}FBE\textnormal{)} of $\varphi$ associated with the parameter $\gamma$ is defined by
\begin{equation}\label{eq:FBE_defenition}
    \varphi_{\gamma}(x):=\inf_{y \in \mathbb{R}^n} \{h(x)+\langle \nabla h(x),y-x \rangle +g(y) + \frac{1}{2\gamma} \|y-x\|^2 \},
\end{equation}
which can be represented using Moreau envelope \eqref{eq:Moreau_Envelope} as follows:
\begin{equation}
     \varphi_{\gamma}(x)=h(x)-\frac{\gamma}{2}\|\nabla h(x)\|^2+e_{\gamma}g(x-\gamma\nabla h(x)).
\end{equation}

\end{definition}

The FBE has already been used for developing some efficient algorithms to solve nonsmooth optimization problems; see, e.g., \cite{patrinos2013proximal,stella2017forward,themelis2018forward} and references therein. We recall several properties of the FBE for convex composite extended-real-valued functions, adapted from \cite{patrinos2013proximal, stella2017forward}, that are required for the implementation of our proposed algorithm.

\begin{proposition}[\textbf{Several properties of FBE}]\label{prop:several_properties_FBE}
   Let $\varphi=h+g$ be as defined in \eqref{eq:composite_convex_definition}, and let $\gamma>0$. Assume that $h$ is $C^2$-smooth on $\mathbb{R}^n$ and that $\nabla h$ is Lipschitz continuous on $\mathbb{R}^n$ with modulus $l>0$. Then, the following properties hold:\\
    $(i)$ The FBE $\varphi_{\gamma}$ of $\varphi$ is $C^1$-smooth on $\mathbb{R}^n$ with the gradient 
    \[
    \nabla \varphi_{\gamma}(x)=\gamma^{-1}(I-\gamma \nabla^2h(x))(x-\text{Prox}_{\gamma g}(x-\gamma \nabla h(x))), \  \ \  x\in \mathbb{R}^n.
    \]\\
    Moreover, the set of optimal solutions to \eqref{eq:composite_convex_definition}
coincides with the set of stationary points of $\varphi_\gamma$; in particular,
    \[
    \text{argmin } \varphi= \text{zer} \nabla \varphi_\gamma:=\{x \in \mathbb{R}^n\  |\ \nabla \varphi_\gamma(x)=0\} \ \ \text{for all} \ \ \gamma\in (0,1/l).
    \]
    $(ii)$ Let $h(x):=\frac{1}{2}\langle Ax,x \rangle+\langle b,x \rangle+\alpha$, where $A\in \mathbb{R}^{n\times n}$ is a positive-semidefiite symmetric matrix, $b \in \mathbb{R}^n$, and $\alpha \in \mathbb{R}$.  Then for all $\gamma \in (0,1/l)$, the FBE $\varphi_\gamma$ is convex and its gradient $\nabla \varphi_{\gamma}$ is globally Lipshitzian on $\mathbb{R}^n$ with modulus $L$, where $ L:=2(1-\gamma \lambda_{min(A)})/{\gamma}$. If $A$ is positive-definite, then $\varphi_\gamma$ is strongly convex with modulus $K$, where
    \[
    K:=\min \{(1- \gamma \lambda_{min (A)})\lambda_{min(A)}, (1-\gamma\lambda_{max(A)})\lambda_{max(A)}\}.
    \]
   
\end{proposition}

In the remainder of this section, we consider three representative convex composite optimization problems arising in machine learning: the lasso, support vector machine (SVM), and logistic lasso. These models capture important applications in sparse learning and classification and serve to illustrate the applicability of the FBE and the proposed coderivative-based Newton framework to a broad class of nonsmooth machine learning problems.

\subsection{Lasso}\label{sec:lasso_sub_section}

In this section, we consider the Lasso problem as a particular instance of a convex composite optimization problem with a quadratic smooth component. Specifically, we study the problem
\begin{equation}\label{eq:quadratic_convex_composite}
    \min_{x\in\mathbb{R}^n} \;
    \varphi(x)
    :=\frac{1}{2}\langle Ax,x\rangle+\langle b,x\rangle+\alpha+g(x),
\end{equation}
where $A\in\mathbb{R}^{n\times n}$ is a symmetric positive semidefinite matrix, $b\in\mathbb{R}^n$, $\alpha\in\mathbb{R}$, and $g:\mathbb{R}^n\to\overline{\mathbb{R}}$ is a convex, extended real-valued function.

A prominent example of \eqref{eq:quadratic_convex_composite} is the Lasso problem, originally introduced by Tibshirani \cite{tibshirani1996regression}. Given a data matrix $\mathcal{A}\in\mathbb{R}^{m\times n}$, an observation vector $\mathcal{Y}\in\mathbb{R}^m$, and a regularization parameter $\mu>0$, the Lasso problem is formulated as
\begin{equation}\label{eq:lasso_first_formulation}
    \min_{x\in\mathbb{R}^n}
    \phi(x)
    :=\frac{1}{2}\|\mathcal{A}x-\mathcal{Y}\|_2^2+\mu\|x\|_1.
\end{equation}
The least-squares term measures the discrepancy between the observations and the model predictions, whereas the $\ell_1$-regularization term promotes sparse solutions. This sparsity-inducing property has made the Lasso a fundamental model in high-dimensional statistics and machine learning, as well as in a variety of applications involving sparse estimation.

Expanding the quadratic data-fitting term in \eqref{eq:lasso_first_formulation} yields
\[
    \frac{1}{2}\|\mathcal{A}x-\mathcal{Y}\|_2^2
    =
    \frac{1}{2}\langle \mathcal{A}^*\mathcal{A}x,x\rangle
    -\langle \mathcal{A}^*\mathcal{Y},x\rangle
    +\frac{1}{2}\|\mathcal{Y}\|_2^2.
\]
Consequently, \eqref{eq:lasso_first_formulation} can be equivalently expressed as
\begin{equation}\label{eq:lasso_quadratic}
    \phi(x)
    =
    \frac{1}{2}\langle \widetilde{A}x,x\rangle
    +\langle \widetilde{b},x\rangle
    +\widetilde{\alpha}
    +\mu\|x\|_1,
\end{equation}
where
\[
    \widetilde{A}:=\mathcal{A}^*\mathcal{A},
    \qquad
    \widetilde{b}:=-\mathcal{A}^*\mathcal{Y},
    \qquad
    \widetilde{\alpha}:=\frac{1}{2}\|\mathcal{Y}\|_2^2.
\]
Since $\widetilde{A}$ is symmetric positive semidefinite, \eqref{eq:lasso_quadratic} is a convex composite problem of the form \eqref{eq:quadratic_convex_composite}, with
\[
    g(x)=\mu\|x\|_1.
\]
Thus, the Lasso problem falls directly within the class of problems considered in this paper.

Although the objective function in \eqref{eq:lasso_first_formulation} is convex, it is generally nonsmooth due to the $\ell_1$-regularization term. To facilitate the development of smooth optimization methods, we employ FBE function, which provides a continuously differentiable reformulation while preserving the minimizers of the original problem.

The following three propositions, adapted from \cite{khanh2024globally}, are useful for applying the NCN algorithm to the FBE associated with problem \eqref{eq:quadratic_convex_composite}. Firstly, the following proposition provides an explicit characterization of the coderivative of the FBE function.

\begin{proposition}\label{prop:quadratic_convec_composite_coderivative}
   Let $\varphi=h+g$ be defined as in \eqref{eq:quadratic_convex_composite}, and let $\gamma>0$ be such that $B:=I-\gamma A$ is positive definite. Then, the following formula holds:
    \begin{equation}\label{eq:FBE_coderivatice_relationship}
        \bar{z}\in D^\ast \nabla \varphi_{\gamma}(\bar{x})(\omega) \Longleftrightarrow B^{-1}\bar{z}-A \omega \in D^\ast\partial g (\textnormal{Prox}_{\gamma g}(\bar{u}),\frac{1}{\gamma}(\bar{u}-\textnormal{Prox}_{\gamma g}(\bar{u})))(\omega-\gamma B^{-1}\bar{z})
    \end{equation}
    for any $\bar{x}\in \mathbb{R}^n, \ \omega \in \mathbb{R}^n$, and $\bar{u}:=\bar{x}-\gamma(A\bar{x}+b)$.
\end{proposition}

 The following proposition establishes the equivalence of metric regularity and tilt stability for problem \eqref{eq:quadratic_convex_composite} and its corresponding FBE formulation in \eqref{eq:FBE_defenition}. Moreover, it provides an estimate for the inverse mapping of $\partial \varphi_{\gamma}$ in terms of the problem data in \eqref{eq:quadratic_convex_composite}.
\begin{proposition}\label{prop:quadratic_convex_composite_tilt_stability}
    Let $\varphi=h+g$ be as in (\ref{eq:quadratic_convex_composite}), and let $\gamma>0$ be such that $B:=I-\gamma A$ is positive-definite. Then for any $\bar{x}\in \mathbb{R}^n$ satisfying $0 \in \partial \varphi (\bar{x})$ we have:\\
    (i) $\|\partial^2 \varphi_{\gamma}(\bar{x})^{-1}\| \le \|\partial^2 \varphi(\bar{x},0)^{-1}\|+\gamma\|B^{-1}\|.$\\
    (ii)\  $\partial \varphi \text{ is metrically regular around } (\bar{x},0) \text{ if and only if } \nabla \varphi_{\gamma}$ \text{is metrically regular around} $\bar{x}$.\\
    (iii)\ $\bar{x}$ is a tilt-stable local minimizer of $\varphi$ if and only if $\bar{x}$ is a tilt-stable local minimizer of $\varphi_\gamma$.\\
\end{proposition}
The following proposition is used to establish the superlinear convergence of the NCN algorithm when applied to the FBE associated with problem \eqref{eq:quadratic_convex_composite}.

\begin{proposition}\label{prop:semismooth_star_quadratic_convex}
    Let $\varphi=h+g$ be as in (\ref{eq:quadratic_convex_composite}), and let $\gamma>0$ be such that $B:=I-\gamma A$ is positive-definite. Then for any $\bar{x}\in \mathbb{R}^n$ satisfying the stationary condition $0 \in \partial \varphi (\bar{x})$ the following assertions hold:\\
    (i) $\nabla \varphi_\gamma$ is $semismooth^\ast$ at $\bar{x}$ if $\partial g$ is $semismooth^\ast$ at $(\bar{x},\bar{v})$, where $\bar{v}:=-A\bar{x}-b$.\\
    (ii) $\nabla \varphi_\gamma$ is directionally differentiable at $\bar{x}$ if $g$ is twice epi-differentiable at $\bar{x}$ for $\bar{v}$.
\end{proposition}

Propositions \ref{prop:quadratic_convec_composite_coderivative}, \ref{prop:quadratic_convex_composite_tilt_stability}, and \ref{prop:semismooth_star_quadratic_convex} provide the necessary characterizations for applying the NCN algorithm to the FBE associated with problem \eqref{eq:quadratic_convex_composite}.

In the remainder of this section, we focus on the lasso problem~\eqref{eq:lasso_quadratic} which is an important example of the convex composite problem with quadratic smooth part. To solve the lasso problem, we focus on the following problem 
\begin{equation}\label{eq:lasso_FBE_minimization}
    \min \phi_{\gamma}(x) \ \ s.t. \ \ x\in\mathbb{R}^n
\end{equation}
where $\phi_\gamma (x)$ is the FBE associated with problem~\eqref{eq:lasso_quadratic}. By Proposition~\ref{prop:several_properties_FBE}, $\phi_\gamma$ is a $C^{1,1}$ convex function with the same set of minimizers as $\phi$ in~\eqref{eq:lasso_first_formulation}. Hence, Algorithm~\ref{alg:nomonotone-newton} can be applied to solve the minimization problem~\eqref{eq:lasso_FBE_minimization}. To do this, in the first step, we need an explicit formulation to solve the inclusion (\ref{eq:direction_inclusion_nonmonotone}) which is here as
\begin{equation}\label{eq:lasso_application_direction_inclusion}
    -\nabla \phi_\gamma (x) \in D^\ast \nabla \phi_\gamma(x)(d)+\mu d
\end{equation}
where we set $R=I$, and $I$ is an identity matrix. In this regard, Proposition~\ref{prop:quadratic_convec_composite_coderivative} shows that we need to compute the coderivative of the subdifferential mapping \(\partial g\), where \(g(x)=\mu\|x\|_1\) is the regularization function. Following~\cite{khanh2024globally}, we present an explicit formula for \(D^\ast\partial g\) and subsequently derive the equation for computing the search direction of the NCN algorithm applied to \(\phi_\gamma\). For the regularization function $g$, we have

\begin{equation}
    \big(\textnormal{Prox}_{\gamma g}(z)\big)_i= 
    \begin{cases}
    z_i-\mu \gamma& \  \textnormal{if } z_i > \mu \gamma,\\
    0              & \  \textnormal{if } - \mu \gamma \le z_i \le \mu \gamma, \\
    z_i+\mu \gamma & \textnormal{if } z_i<- \mu \gamma.
\end{cases}
\end{equation}
and
\begin{equation}\label{eq:subdifferential_of_g_lasso}
\partial g(z)
=
\left\{
v\in\mathbb{R}^n
\;\middle|\;
\begin{array}{lc}
v_i=\mu\,\operatorname{sgn}(z_i), & \text{if } z_i\neq 0,\\[2mm]
v_i\in[-\mu,\mu], & \text{if } z_i=0,
\end{array}
\quad i=1,\ldots,n
\right\} \quad \textnormal{for}\ z\in\mathbb{R}^n.
\end{equation}
Then, the generalized Hessian of $g$ is given by 
\begin{equation}
    D^\ast\partial g(z,y)(v)=\{\omega \in \mathbb{R}^n \ |\ (\frac{1}{\mu}\omega_i,-v_i)\in G(z_i, \frac{1}{\mu}y_i),\ i=1,2,\dots,n\}
\end{equation}
for each $(z,y)\in \textnormal{gph}\  \partial g$ and $v=(v_1,v_2, \dots, v_n)\in \mathbb{R}^n$, where the mapping $G:\mathbb{R}^2 \rightrightarrows \mathbb{R}^2$ is defined by\\
\begin{equation}
G(t,p)=\begin{cases}
    \{0\}\times \mathbb{R}& \  \textnormal{if } t\neq0, \ p\in \{-1,1\},\\
    \mathbb{R}\times \{0\}              & \  \textnormal{if } t=0, p \in (-1,1), \\
    (\mathbb{R}_{+} \times \mathbb{R}_{-})\cup (\{0\} \times \mathbb{R}) \cup (\mathbb{R}\times \{0\}) & \textnormal{if } t=0, p=-1,\\
     (\mathbb{R}_{-} \times \mathbb{R}_{+})\cup (\{0\} \times \mathbb{R}) \cup (\mathbb{R}\times \{0\}) & \textnormal{if } t=0, p=1,\\
     \emptyset & \textnormal{otherwise}.
\end{cases}
\end{equation}
Now, we describe in detail how the search direction \(d\) is computed. To this end, we have

\begin{equation}
    (v^k)_i=
    \begin{cases}
        (u^k)_i-\mu \gamma & \ \textnormal{if}\ (u^k)_i>\mu \gamma,\\
        0 & \textnormal{if} -\mu \gamma \le (u^k)_i \le \mu \gamma, \ \ \textnormal{where } u^k=x_k-\gamma(\mathcal{A}^*\mathcal{A}x_k+\mathcal{Y}).\\
        (u^k)_i+\mu \gamma & \textnormal{if } (u^k)_i<- \mu \gamma
    \end{cases}
\end{equation}
 where
\begin{equation}
        (X^k)_i:=\begin{cases}
            \gamma (\mathcal{A}^*\mathcal{A})_i & \textnormal{ if }(v^k)_i \neq 0,\\
            I_i & \textnormal{ if }(v^k)_i=0.       
        \end{cases}
\end{equation}
At each iteration \(k\) of the NCN algorithm, we compute the search direction \(d_k\) by solving the following equation:

\begin{equation}
    (B X^k+\gamma \mu_k I)d_k=B(v^k-x_k),
\end{equation}
where $B:=I-\gamma \mathcal{A}^\ast \mathcal{A}$.

As a final step, it suffices to verify that problem \eqref{eq:lasso_FBE_minimization} satisfies the conditions required to establish the global convergence and local superlinear convergence rate of the NCN algorithm. Since $I$ is a positive definite symmetric matrix and $D^\ast\nabla\phi_\gamma(x_k)$ is positive semidefinite due to the convexity of $\phi_\gamma$ on its domain, global convergence follows from Proposition~\ref{prop:global_convergent_nonmonotone}. Moreover, by \eqref{eq:subdifferential_of_g_lasso}, the graph of $\partial g$ can be represented as a finite union of closed convex sets. Hence, $\partial g$ is semismooth$^\ast$ on its graph; see \cite{gfrerer2021semismooth}. Furthermore, $g$ is a proper, convex, piecewise linear function on $\mathbb{R}^n$. It then follows from Proposition 13.9 in \cite{rockafellar1998variational} that $g$ is twice epi-differentiable on $\mathbb{R}^n$. Under the assumption that the minimizer \(\bar{x}\) is tilt-stable, Proposition 4.5 in~\cite{drusvyatskiy2014second} ensures the metric regularity of \(\partial\phi\) around \((\bar{x},0)\). Consequently, by applying Propositions~\ref{prop:semismooth_star_quadratic_convex}, \ref{prop:quadratic_convex_composite_tilt_stability} and~\ref{prop:superlinear_convergent_rate}, we obtain the local superlinear convergence rate of the NCN algorithm for problem \eqref{eq:lasso_FBE_minimization}.

\subsection{Support Vector Machine}\label{sec:application_svm}
This section considers the support vector machine (SVM) model, which is widely used for binary classification problems; see (\cite{gareth2013introduction}, Chapter 6) and the references therein. Let $\mathcal{S}\in\mathbb{R}^{m\times n}$ be a data matrix and let $y^{\mathrm{svm}}\in\mathbb{R}^m$ denote the corresponding label vector, where $y_i^{\mathrm{svm}}\in\{1,-1\}$ for $i=1,\ldots,m$. The SVM approach constructs a linear classifier by solving the following optimization problem:
\begin{equation}\label{eq:svm-primal}
\begin{aligned}
\mbox{minimize }\;\quad&
\frac{1}{2}\|w\|^2 + C\sum_{i=1}^m \xi_i
\;\mbox{ over }\;w,\bar{b},\mathbf{\xi}\\
\mbox{subject to }\quad&
y^{\mathrm{svm}}_i(\langle w,\bar{a}_i\rangle+\bar{b})
\ge 1-\xi_i,\quad
\xi_i\ge 0,\quad i=1,\ldots,m,
\end{aligned}
\end{equation}
where $\bar{a}_i$ denotes the $i$-th row of the data matrix $\mathcal{S}$. The corresponding dual problem of \eqref{eq:svm-primal} is given by
\begin{equation}\label{eq:svm-dual}
\begin{aligned}
\mbox{minimize }\;\quad&
\frac{1}{2}\langle \mathbf{\lambda}, D\mathbf{\lambda}\rangle
-\langle \mathbf{\lambda}, e \rangle
\;\mbox{ over }\;\mathbf{\lambda}\in\mathbb{R}^m\\
\mbox{subject to }\;\quad&
\langle \mathbf{\lambda}, y^{\mathrm{svm}}\rangle=0,\quad
0\le \mathbf{\lambda}\le Ce,
\end{aligned}
\end{equation}
where $e\in\mathbb{R}^m$ is the all-one vector, and $D\in\mathbb{R}^{m\times m}$ is a symmetric positive semidefinite matrix whose entries are defined by
\[
D_{ij}
=
y_i^{\mathrm{svm}}y_j^{\mathrm{svm}}
\langle\bar{a}_i,\bar{a}_j\rangle.
\]
For nonlinear classifiers, the dual SVM problem can be formulated in the same manner by using an appropriate kernel function.

The dual problem \eqref{eq:svm-dual} is a convex quadratic programming problem with a single linear equality constraint and bound constraints on the variables. In particular, its feasible region is the intersection of a hyperplane and a box. This observation motivates the consideration of the following more general class of convex quadratic programs with a single linear constraint and bound constraints, which we abbreviate as SLBQP:
\begin{equation}\label{eq:SLBQP}
\begin{aligned}
\mbox{minimize }&
\frac{1}{2}\langle x,Qx \rangle+\langle c,x \rangle
\mbox{ over }x\in\mathbb{R}^n\
\quad \mbox{subject to }&
\langle p,x \rangle=q,\quad l\le x\le u,
\end{aligned}
\end{equation}
where $Q\in\mathbb{R}^{n\times n}$ is a symmetric positive semidefinite matrix, $c\in\mathbb{R}^n$, $p\in\mathbb{R}^n$, $q\in\mathbb{R}$, $l\in(\mathbb{R}\cup{\{-\infty\}})^n$, and $u\in(\mathbb{R}\cup{\{\infty \}})^n$. The SLBQP class includes support vector machine problems and bound-constrained quadratic programming problems as special cases.

We observe that problem~(\ref{eq:SLBQP}) can be formulated as a convex composite problem~(\ref{eq:composite_convex_definition}) with

\begin{equation}\label{eq:convex_composite_SLBQP}
\min h(x)+g(x) \quad \text{where} \quad h(x)=\frac{1}{2}\langle x,Qx\rangle+\langle c,x\rangle
\quad\text{and}\quad
g(x)=\delta_{\Omega}(x),
\end{equation}

where $\delta_{\Omega}$ denotes the indicator function of the feasible set

$$
\Omega:=\{x\in\mathbb{R}^n\mid \langle p,x\rangle=q,\; l\leq x\leq u\}.
$$

This formulation allows us to employ the FBE and define the function $\varphi_\gamma$ in~(\ref{eq:FBE_defenition}). By Proposition~\ref{prop:several_properties_FBE}, $\varphi_\gamma$ is $C^{1,1}$ and has the same set of minimizers as the original problem. Consequently, the NCN algorithm can be applied to $\varphi_\gamma$. To implement the NCN method, we need to characterize the coderivative of $\nabla\varphi_\gamma$. By Proposition~\ref{prop:quadratic_convec_composite_coderivative}, this coderivative can be characterized in terms of the coderivative of $\partial g$. Based on this characterization, we derive an explicit Newton system for computing the search direction of the NCN algorithm. Recall that, after constructing the forward-backward envelope $\varphi_\gamma$ associated with the convex composite problem in~(\ref{eq:convex_composite_SLBQP}), where $h$ and $g$ are defined therein, the computation of the Newton direction requires solving the inclusion in~(\ref{eq:direction_inclusion_nonmonotone}). According to Proposition~\ref{prop:quadratic_convec_composite_coderivative}, this requires, first, evaluating the proximal mapping of $g$ and, subsequently, the coderivative of $\partial g$ at the corresponding proximal point.

Since $g$ is the indicator function of the convex polyhedral set $\Omega$, its proximal mapping is given by the Euclidean projection onto $\Omega$, i.e.,

$$
\operatorname{Prox}_{\gamma g}(x)=\operatorname{Proj}_{\Omega}(x).
$$
Moreover, the limiting subdifferential of $g$ is characterized by the normal cone to $\Omega$ as

$$
\partial g(x)=N_{\Omega}(x), \qquad x\in\Omega,
$$

and $\partial g(x)=\emptyset$ for $x\notin\Omega$. Hence, the coderivative of the subdifferential mapping $\partial g$ coincides with the coderivative of the normal-cone mapping $N_{\Omega}$. Since $\Omega$ is a convex polyhedral set, an explicit characterization of $D^*N_{\Omega}$ is available; see, e.g.,~\cite{chao2024coderivative}.

In the final part of this section, following \cite{chao2024coderivative}, we present the Newton system for computing the search direction. To this end, we first express the feasible set $\Omega$ in the equivalent form
\[
\Omega:=\{x \in \mathbb{R}^n | \langle p,x \rangle=q, \ l \le x \le L\}= \{x \in \mathbb{R}^n \ |\ \langle p,x \rangle=q, Cx \le \bar{d}\}, 
\]
where $p\in\mathbb{R}^n$ and $l,L\in\overline{\mathbb{R}}^n$. Without loss of generality, we assume that
$l_i<L_i$ for all $i=1,\ldots,n$. The matrix $C$ and vector $\bar d$ are given by
\[ C = \begin{pmatrix} I_n\\ -I_n \end{pmatrix}\quad \textnormal{and} \quad \bar{d}= \begin{pmatrix} L\\ -l \end{pmatrix}. \]
Next, for $u\in\mathbb{R}^n$, we define the matrix $P$ associated with the projection $\operatorname{Proj}_{\Omega}(u)$.

\begin{equation*}
P:=\left\{\begin{aligned}
&\Sigma &\text{ if}\quad \Sigma p =0,\\
&\Sigma-\|\Sigma a\|^{-2}\Sigma p p^T \Sigma &\text{ if}\quad \Sigma p\neq 0,
\end{aligned}\right.
\end{equation*}
where $\Sigma:=I-\Theta$, with $\Theta:=\operatorname{diag}(\theta)$, and

\[\theta_i:=\left\{
\begin{aligned}
&0\;\textrm{ if}\quad l_i<\textnormal{Proj}_\Omega(u)_i<L_i,\\
&1\;\textrm{ otherwise}.
\end{aligned}
\right.\]

Now by coosing regularization matrix $R_k=R:=I-\gamma Q  \succ 0$, the Newton system to find a search direction is\\
\[
((1+\mu_k)I-P^kR)d_k=\textnormal{Proj}_\Omega(u^k)-x_k, \quad \textnormal{where}\quad u^k:=x_k-\gamma (Qx_k+c)
\]
and where $P^k$ is the matrix $P$  generated by $u^k$. We can write\\
$\bullet$ If $\Sigma p=0$, then we have
\begin{equation}
\left\{\begin{aligned}
&(1+\mu_k)(d_k)_\delta= (v^k-x^k)_\delta,\\
&(1+\mu_k)I_{\bar\delta}-R_{\bar\delta\times\bar\delta}(d_k)_{\bar\delta}=(v^k-x^k)_{\bar\delta}+R_{\bar\delta\times\delta}(d_k)_\delta,
\end{aligned}
\right.
\end{equation}
where $v^k:=\textnormal{Proj}_\Omega(u^k)$, $\delta:=\{i\;|\;v_i^k= l_i\;
\text{ or }\;v_i^k= L_i\}$, and $\bar\delta=[n]-\delta=\{i\;|\;l_i<v_i^k<L_i\}$.\\[1ex]
$\bullet$ If $\Sigma p \neq 0$, then we have
\begin{equation}
\left\{\begin{aligned}
&(1+\mu_k)(d_k)_\delta=(v^k-x^k)_\delta,\\
&(1+\mu_k)I_{\bar\delta}-P^k_{\bar\delta\times\bar\delta}R_{\bar\delta\times \bar\delta})(d_k)_{\bar\delta}=(v^k-x^k)_{\bar\delta}+P^k_{\bar\delta\times\bar\delta}R_{\bar\delta\times\delta}(d_k)_{\delta}.
\end{aligned}
\right.
\end{equation}\vspace*{-0.15in}

\subsection{Logistic Lasso}\label{sec:logistic_lasso_application}
In this section, we consider the logistic Lasso problem, which is a well-known example of a convex composite optimization problem with a nonquadratic smooth component. More generally, we consider problem \eqref{eq:composite_convex_definition}, where the smooth function $h(x)$ is not necessarily quadratic. To address this broader class of problems, we adopt a standard approach from the literature in which the smooth component $h$ is locally approximated by a quadratic function based on its Taylor expansion. This approach allows us to extend the analysis developed for quadratic convex composite problems to a more general class of problems with nonquadratic smooth components. In particular, it leads to a class of second-order methods in which, at each iteration $k$, the next iterate $x_{k+1}$ is obtained by solving the following quadratic composite subproblem:

\begin{equation}\label{eq:subproblem_Proximal_Newton}
    x_{k+1}=\text{arg}\min_{x \in \mathbb{R}^n} h(x_k)+\langle \nabla h(x_k), x-x_k \rangle+\frac{1}{2} \langle x-x_k, \nabla^2h(x_k) (x-x_k) \rangle +g(x)
\end{equation}

The use of such a quadratic model for the smooth component forms the basis of the proximal Newton method; see, e.g., \cite{lee2014proximal,kanzow2021globalized}. By omitting the terms that are constant with respect to $x$, the subproblem in \eqref{eq:subproblem_Proximal_Newton} can equivalently be expressed as
\begin{equation}\label{eq:cleanest_subproblem}
x_{k+1}\in\underset{x\in\mathbb{R}^n}{\operatorname{argmin}};
\frac{1}{2}\left\langle x,\nabla^2h(x_k)x\right\rangle
+\left\langle \nabla h(x_k)-\nabla^2h(x_k)x_k,x\right\rangle
+g(x).
\end{equation}

Since the subproblem \eqref{eq:cleanest_subproblem} has a quadratic smooth component and retains the nonsmooth term $g$, it falls within the class of quadratic convex composite problems considered in Section~\ref{sec:lasso_sub_section}. Hence, the NCN algorithm developed in that section can be employed to efficiently solve the proximal Newton subproblem. Incorporating the NCN procedure into the proximal Newton framework yields the Coderivative-based Proximal Newton (CPN) method, which is summarized in Algorithm~\ref{alg:Proximal_Newton}.

\begin{algorithm}
  \caption{A Coderivative-based Proximal Newton (CPN) Method}
    \label{alg:Proximal_Newton}
\begin{algorithmic}[1]
\State \textbf{Input:} $x_0 \in \mathbb{R}^n$, $0<\beta' <1$, $\bar{c}\in(0,0.5)$, $\theta>0$, tol$>0$, $K\in \mathbb{N}$ and all parameters required by Algorithm~\ref{alg:nomonotone-newton}.
 \State $r_{0}=\frac{\|x_{0}-\text{Prox}_{\theta g}(x_{0}-\theta \nabla h(x_{0}))\|}{1+ \|x_{0}\|}$
\While{$k = 0,1,2,\dots K$ \ \text{and} $r_k >\text{tol}$}
    \If{$\nabla f(x_k)=0$}
        \State \textbf{stop}
    \Else
       \State $x^{*}_{\mathrm{sub}} \leftarrow$ Solve the subproblem \eqref{eq:cleanest_subproblem} using the NCN algorithm (Algorithm~\ref{alg:nomonotone-newton}).
    \EndIf
    \State $d_k=x^{*}_{\text{sub}}-x_k$
    \State $\alpha_k = 1$
    \While{$ \varphi(x_k+\alpha_k d_k) \le \varphi(x_k)+\bar{c} \alpha_k (\langle \nabla h(x_k), d_k \rangle+ g(x_k+d_k)- g(x_k))$}
    \State $\alpha_k=\beta^\prime \alpha_{k-1}$
    \EndWhile
    \State Set $x_{k} := x_k + \alpha_k d_k$
    \State $r_{k}=\frac{\|x_{k}-\text{Prox}_{\theta g}(x_{k}-\theta \nabla h(x_{k}))\|}{1+ \|x_{k}\|}$

\EndWhile

\State \textbf{return} $x_k$

\end{algorithmic}
\end{algorithm}

As noted above, Algorithm~\ref{alg:Proximal_Newton} is applicable to problem \eqref{eq:composite_convex_definition} when the smooth component $h$ is nonquadratic. At each iteration, the CPN method constructs the quadratic subproblem \eqref{eq:cleanest_subproblem} based on the second-order information of $h$ at the current iterate and employs the NCN algorithm to solve the resulting quadratic convex composite problem. The solution of this subproblem determines the search direction for the CPN method, while the step size $\alpha_k$ is selected using a backtracking line search procedure adapted from \cite{lee2014proximal}. Thus, the proposed framework combines the proximal Newton approximation of the nonquadratic smooth component with the coderivative-based Newton procedure developed for quadratic convex composite problems.

In the remainder of this section, we specialize the proposed CPN method to the logistic Lasso problem, which is a standard binary classification model with $\ell_1$-regularization. Let $\mathcal{D}\in\mathbb{R}^{m\times n}$ denote the data matrix, let $x\in\mathbb{R}^n$ be the parameter vector, and let $y^{\textnormal{lr}}\in\{-1,1\}^m$ be the corresponding vector of binary labels. The logistic Lasso problem can be formulated as
\begin{equation}\label{eq:logistic_lasso_convex_composite}
    F(x)=h(x)+g(x)
    =\frac{1}{m}\sum_{i=1}^{m}\log(1+e^{-z_i})
    +\mu\|x\|_1,
\end{equation}
where
\[
    z_i=y^{\text{lr}}_i\langle a_i,x\rangle,
\]
and $a_i$ denotes the $i$th row of the matrix $\mathcal{D}$. Here, the smooth component is given by
\[
    h(x)=\frac{1}{m}\sum_{i=1}^{m}\log(1+e^{-z_i}),
\]
and the nonsmooth component is $g(x)=\lambda\|x\|_1$, where $\lambda>0$ is the regularization parameter. The nonquadratic nature of the logistic loss makes this problem a natural setting in which to apply the CPN framework developed above.

To construct the quadratic model required by the CPN method, we first derive the first- and second-order derivatives of the smooth component $h$. The gradient of $h$ is given by
\begin{equation*}
    \nabla h(x)
    =\frac{1}{m}\sum_{i=1}^{m}
    \bigl(\sigma(z_i)-1\bigr)y^{\text{lr}}_i a_i,
    \qquad
    \text{where}\quad
    \sigma(z)=\frac{1}{1+e^{-z}}.
\end{equation*}
Furthermore, the Hessian of $h$ admits the following compact matrix representation:
\[
    \nabla^2 h(x)=\frac{1}{m}\mathcal{D}^{T}W\mathcal{D},
\]
where $W=\operatorname{diag}(w_1,\ldots,w_m)$ is a diagonal matrix whose entries are given by
\[
    w_i=\sigma(z_i)\bigl(1-\sigma(z_i)\bigr),
    \qquad i=1,\ldots,m.
\]
Since $0<w_i\leq 1/4$, the matrix $W$ is positive semidefinite, and consequently $\nabla^2h(x)$ is positive semidefinite for every $x\in\mathbb{R}^n$. Thus, the logistic loss is convex and possesses the second-order structure required by the proximal Newton framework. These expressions allow us to explicitly construct the quadratic approximation of $h$ at each iterate and, consequently, formulate the corresponding quadratic subproblem to which the NCN algorithm can be applied.

\section{Numerical Results}\label{sec:experiments}

In this section, we report numerical experiments to assess the computational performance of the proposed NCN and CPN methods. We consider three well-established machine learning optimization problems: Lasso, SVM, and logistic Lasso. These problems provide representative test cases for evaluating the proposed methods in different convex composite optimization settings.

The experiments were performed under fixed computational environments within each problem class to facilitate a consistent comparison. For the Lasso experiments in Section~\ref{sec:lasso_experimental}, all methods were implemented in MATLAB R2023b and tested on a desktop computer equipped with an Intel Core i7-7700 processor and 32 GB of RAM. For the SVM and logistic Lasso experiments in Sections~\ref{sec:SVM-experimental} and \ref{sec:logistic-lasso-experiment}, respectively, the computations were performed in MATLAB R2018a on a desktop computer equipped with an Intel Core i5-10505 processor and 32 GB of RAM.

\subsection{Lasso}\label{sec:lasso_experimental}

Coderivative-based Newton methods constitute a relatively recent class of algorithms in the optimization literature. The first method of this type was introduced by \cite{mordukhovich2021generalized}. Subsequently, \cite{khanh2024globally} proposed an enhanced variant, referred to as the Coderivative-Based Generalized Regularized Newton Method (GRNM), which employs an Armijo line search strategy to establish global convergence. More recently, \cite{chao2024coderivative} introduced another variant, known as the Generalized Regularized Newton Method with Wolfe line search (GRNM-W), which uses the Wolfe line search strategy as its globalization mechanism.

As a first step, we compare the proposed NCN method with existing coderivative-based Newton methods from the literature on the LASSO problem. The application of the NCN algorithm to the LASSO problem is described in Section~\ref{sec:lasso_sub_section}. For all experiments, we use the KKT residual proposed in \cite{li2018highly} as the stopping criterion. Specifically, the algorithm terminates when $r_k<10^{-6}$, where
\begin{equation}
r_k :=
\frac{\|x^k-\operatorname{Prox}_{\mu \|\cdot\|_1}\big(x^k-\mathcal{A}^\ast(\mathcal{A}x^k-\mathcal{Y})\big)\|}
{1+\|x^k\|+\|\mathcal{A}x^k-\mathcal{Y}\|}.
\end{equation}

\begin{table}[ht]
\centering
\caption{Computational results for coderivative-based Newton methods on the Lasso problem.}
\label{tab:results-coderivative-based-method}
\begin{tabular}{lcccccccc}
\toprule
\multirow{2}{*}{Case $(m,n,\mu)$}
& \multicolumn{4}{c}{Iterations}
& \multicolumn{4}{c}{CPU Time (s)} \\
\cmidrule(lr){2-5}
\cmidrule(lr){6-9}
& GRNM & GRNM-W & TR-NCN & NCN
& GRNM & GRNM-W & TR-NCN & NCN \\
\midrule
$(500,1000,10^{-3})$
& \text{2616} & 2862 & 3956 & 3632
& \text{81.4} & 567.0 & 146.2 & 109.1 \\
$(10^2,10^2,10^{-3})$
& 22 & 20 & 17 & \text{15}
& 0.052 & 0.2 & 0.059 & \text{0.047} \\

$(10^3,10^3,10^{-3})$
& 62 & 123 & 38 & \text{37}
& 1.40 & 9.46 & 0.96 & \text{0.83} \\

$(10^3,10^3,10^{-2})$
& 67 & 70 & 58 & \text{52}
& 1.68 & 6.93 & 1.63 & \text{1.29} \\

$(10^3,10^3,10^{-1})$
& 64 & 64 & \text{61} & 64
& 1.89 & 7.30 & 1.75 & \text{1.68} \\

$(5000,10^3,10^{-3})$
& 6 & \text{5} & 6 & \text{5}
& 0.32 & 0.70 & 0.32 & \text{0.27} \\

$(2500,2500,10^{-3})$
& 81 & 75 & \text{59} & 125
& 14.8 & 71.8 & \text{13.2} & 28.4 \\

$(2500,2500,10^{-2})$
& 83 & 91 & 85 & \text{66}
& 15.8 & 98.1 & 18.3 & \text{12.9} \\

$(2500,2500,10^{-1})$
& 92 & 85 & 103 & \text{70}
& 18.8 & 94.2 & 23.1 & \text{13.8} \\
$(5000,2500,10^{-3})$
& 8 & 7 & 7 & \text{6}
& 2.4 & 4.7 & 2.3 & \text{1.9} \\

$(5000,2500,10^{-2})$
& 9 & 8 & 7 & \text{6}
& 2.7 & 6.3 & 2.3 & \text{1.9} \\

$(5000,2500,10^{-1})$
& 10 & 9 & 11 & \text{7}
& 2.9 & 7.2 & 3.2 & \text{2.1} \\

$(10000,7500,10^{-3})$
& 10 & 9 & 9 & \text{8}
& 43.8 & 89.0 & 42.2 & \text{37.2} \\
$(10000,5000,10^{-3})$
& 7 & \text{6} & 7 & \text{6}
& 11.5 & 23.7 & 12.0 & \text{10.5} \\

$(5000,5000,10^{-3})$
& 93 & 90 & \text{51} & \text{51}
& 93.5 & 602.1 & 53.6 & \text{51.3} \\

$(5000,5000,10^{-2})$
& 137 & 142 & 149 & \text{117}
& 145.9 & 1099.6
& 171.8 & \text{126.0} \\

$(7500,5000,10^{-2})$
& 10 & 9 & 8 & \text{7}
& 14.4 & 30.1
& 12.1 & \text{10.7} \\

$(7500,5000,10^{-1})$
& 14 & 13 & 10 & \text{10}
& 17.9 & 53.8
& 14.3 & \text{13.7} \\

\bottomrule
\end{tabular}

\vspace{2mm}

\end{table}

The NCN algorithm by default employs the adaptive regularization mechanism given in (\ref{eq:Adaptive_reg_rule_linesearch}). We use the notation TR-NCN to denote the variant of the NCN algorithm in which the adaptive regularization mechanism is instead defined according to (\ref{eq:Adaptive_reg_rule_TR}). For both NCN and TR-NCN, the HAN line search strategy (Algorithm~\ref{alg:HAN_linesearch2}) is implemented with the parameters

\[
C_1=0.5,\qquad C_2=0,\qquad \beta=0.5,\qquad M_1=2.
\]
With this choice, the HAN line search initially operates as an Armijo line search and switches to a nonmonotone line search strategy with a memory size of two once the step size falls below $\frac{1}{2}$. This adaptive switching mechanism allows the algorithms to accept larger step sizes when appropriate, thereby reducing unnecessary backtracking and improving the overall computational efficiency of the NCN method.

\begin{table}[ht]
\centering
\caption{Computational comparison of the proposed methods with existing algorithms for the Lasso problem ($\mu=10^{-3}$).}
\label{tab:lasso-well-known-algorithms-comparsion}

\begin{tabular}{lcccccccccc}
\toprule
\multirow{2}{*}{Case $(m,n)$}
& \multicolumn{5}{c}{Iterations}
& \multicolumn{5}{c}{CPU Time (s)} \\
\cmidrule(lr){2-6}
\cmidrule(lr){7-11}
& SSNAL & FISTA & ADMM & TR-NCN & NCN
& SSNAL & FISTA & ADMM & TR-NCN & NCN \\
\midrule

$(10000,7500)$
& 27 & 1135 & 352 & 9 & \text{8}
& 173.2 & 250.1 & 36.02 & 38.3 & \text{33.5} \\

$(10000,5000)$
& 18 & 297 & 462 & 7 & \text{6}
& 23.6 & 23.1 & 23.8 & 12.1 & \text{10.4} \\

$(10000,2500)$
& 12 & 93 & 511 & 6 & \text{5}
& 2.7 & 7.5 & 9.3 & 2.3 & \text{1.9} \\

$(5000,5000)$
& 1312 & 10000 & 5170 & \text{51} & \text{51} 
& 4931.87 & 1860 & 353.9 & 56.6 & \text{53.7} \\

$(5000,2500)$
& 11 & 303 & 331 & 7 & \text{6}
& 3.7 & 8.0 & 4.5 & 2.0 & \text{1.8} \\

$(2500,2500)$
& 317 & 10000 & 3378 & \text{59} & 125
& 521.03 & 506.1 & 64.09 & \text{12.2} & 25.5 \\

$(2500,1000)$
& \text{4} & 180 & 259 & 6 & 6
& 0.71 & 0.68 & 0.75 & 0.30 & \text{0.20} \\

$(1000,1000)$
& \text{25} & 58637 & 540 & 38 & 37
& 6.20 & 59.7 & 1.09 & \text{0.95} & 0.98 \\

$(1000,2500)$
& \text{19} & 10000 & 10000 & 6444 & 6496
& \text{4.38} & 247.4 & 106.8 & 1520 & 1377 \\

$(1000,500)$
& \text{4} & 290 & 2000 & 7 & 6
& 0.08 & 0.10 & 1.51 & 0.10 & \text{0.07} \\

$(500,1000)$
& 18 & 17827 & 2000 & 3956 & 3632
& \text{0.82} & 3.94 & 1.53 & 129.8 & 108.1 \\

\bottomrule
\end{tabular}

\vspace{2mm}

\end{table}

We first compare the proposed NCN and TR-NCN methods with the existing coderivative-based Newton methods GRNM and GRNM-W. Since all four methods belong to the coderivative-based Newton framework, this comparison allows us to assess the effect of the proposed adaptive regularization and HAN line search strategies within a common algorithmic setting. The numerical results are reported in Table~\ref{tab:results-coderivative-based-method}. The results demonstrate that the proposed methods are highly competitive with the existing coderivative-based approaches in terms of both iteration count and CPU time. In particular, NCN achieves the lowest CPU time in the majority of the test instances and also requires the fewest iterations in most cases. TR-NCN is particularly effective for several instances, most notably the $(2500,2500,10^{-3})$ case, where it substantially reduces the number of iterations and CPU time compared with GRNM and GRNM-W. Overall, the results indicate that incorporating the HAN line search strategy together with the adaptive regularization mechanism in (\ref{eq:Adaptive_reg_rule_linesearch}) can significantly improve the practical performance of coderivative-based Newton methods.

The advantage of NCN is particularly evident in terms of computational time. For several test instances, NCN requires substantially less CPU time than GRNM; for example, the computational time is reduced by approximately $40\%$ or more in two cases. Although the iteration counts of NCN and GRNM are sometimes comparable, the lower CPU time of NCN suggests that the proposed line search and regularization mechanisms can reduce the computational overhead associated with each iteration. In contrast, GRNM-W generally requires considerably more CPU time than the other coderivative-based methods, despite often using a comparable number of iterations. This behavior is consistent with the additional computational cost associated with the Wolfe line search procedure.

In the second set of experiments, we compare the proposed NCN method with several well-established algorithms for solving the LASSO problem, namely, the Semismooth Newton Augmented Lagrangian (SSNAL) method\footnote{\url{https://www.polyu.edu.hk/ama/profile/dfsun/}}, the Fast Iterative Shrinkage-Thresholding Algorithm (FISTA)\footnote{\url{https://github.com/dattran24/FISTA-lasso}}, and the Alternating Direction Method of Multipliers (ADMM)\footnote{\url{https://web.stanford.edu/~boyd/papers/admm/lasso/lasso.html}}. The results are reported in Table~\ref{tab:lasso-well-known-algorithms-comparsion}, where the maximum number of iterations is set to $10{,}000$ for all methods. The numerical results demonstrate the strong performance of NCN, particularly for instances with $m\geq n$. In these cases, NCN generally requires substantially fewer iterations and achieves competitive, and often the lowest, CPU times among the compared methods. For example, for the $(10000,7500)$ instance, NCN requires only $8$ iterations and $33.5$ seconds, making it more than five times faster than SSNAL. Similar improvements are observed across several other instances. These results highlight the computational efficiency of the proposed coderivative-based Newton framework for square and overdetermined LASSO problems. Finally, following the experimental design of \cite{khanh2024globally}, we evaluated the performance of the considered methods based on the number of iterations and CPU time, with the data matrix $\mathcal{A}$ randomly generated for all numerical experiments reported in Tables~\ref{tab:results-coderivative-based-method} and~\ref{tab:lasso-well-known-algorithms-comparsion}.

\subsection{Support Vector Machine}\label{sec:SVM-experimental}
The support vector machine is a method for binary classification and can be formulated as the convex composite optimization problem in \eqref{eq:convex_composite_SLBQP}. The FBE allows us to reformulate \eqref{eq:convex_composite_SLBQP} as an unconstrained minimization problem with a $C^{1,1}$ objective function. We therefore apply the NCN algorithm, where the search direction $d_k$ is computed according to the procedure described in Section~\ref{sec:application_svm}. The adaptive regularization parameter is updated according to \eqref{eq:Adaptive_reg_rule_linesearch}, and the parameters of the HAN line search are set to
\[
C_1=0.5,\qquad C_2=0,\qquad \beta=0.5,\qquad M_1=2.
\]

Following the experimental setup in \cite{chao2024coderivative}, we generate a symmetric positive semidefinite matrix \(Q\in\mathbb{R}^{n\times n}\) as

$$
Q=W^{T}W,
$$

where \(W\in\mathbb{R}^{r\times n}\) is generated with i.i.d normal distribution entries, and consequently, \(\operatorname{rank}(Q)\) is expected to be $r$. 

Recently, \cite{chao2024coderivative} proposed a coderivative-based Newton method, referred to as the coderivative-based forward-backward Newton (CNFB) method, for solving the SLBQP problem. The authors compared CNFB with several existing approaches, including Gurobi, MATLAB's \texttt{quadprog}, P2GP \cite{di2018two}, and QPPAL \cite{liang2022qppal}. Their numerical results demonstrate that, in the high-rank case, P2GP achieves the best performance, followed by CNFB, whereas Gurobi, MATLAB, and QPPAL are less efficient. In the low-rank case, Gurobi performs best, followed by MATLAB and CNFB. Despite these variations, CNFB remains competitive in both settings, demonstrating its robustness with respect to the rank of the matrix $Q$.
In view of these observations, we compare the proposed NCN method directly with CNFB, and the results are reported in Table~\ref{tab:results-NCN-sVM}. Across the test instances, NCN requires less computational time than CNFB in both low- and high-rank cases. This improvement can be attributed to the combination of adaptive regularization and the HAN line search employed in the NCN framework.

\begin{table}[ht]

\centering
\caption{Numerical results for the SLBQP problem}
\label{tab:results-NCN-sVM}
\begin{tabular}{lcccccc}
\toprule
\multirow{2}{*}{$\big($Rank (r), Size (n)$\big)$}
& \multicolumn{3}{c}{CNFB}
& \multicolumn{3}{c}{NCN} \\
\cmidrule(lr){2-4}
\cmidrule(lr){5-7}
& Iterations & CPU Time(s) & Residual & Iterations & CPU Time(s) & Residual \\
\midrule

(1000,1000) 
& 22 & 1.04 & 1.91e-8 & 17
& 0.45 & 7.67e-7  \\

(900,1000) 
& 24 & 1.04 & 4.55e-10 & 19 
& 0.56 & 3.41e-11  \\

(500,1000) 
& 41 & 1.15 & 1.36e-8 & 32
& 0.64 & 4.8e-7  \\

(100,1000) 
& 326 & 7.19 & 6.22e-7 & 371
& 5.66 & 5.19e-8  \\
(2000,2000) 
& 24 & 3.33 & 1.44e-10 & 17
& 1.96 & 2.33e-9  \\

(1800,2000) 
& 27 & 3.66 & 1.56e-7 & 21 
& 2.39 & 1.45e-10  \\

(1000,2000) 
& 52 & 5.65 & 4.73e-9 & 42
& 3.63 & 2.25e-7  \\

(200,2000) 
& 527 & 46.30 & 2.80e-7 & 557 
& 34.85 & 2.76e-7  \\

(200,2000) 
& 527 & 46.30 & 2.80e-7 & 557 
& 34.85 & 2.76e-7  \\

(5000,5000) 
& 28 & 31.99 & 2.42e-9 & 18 
& 18.68 & 3.31e-9  \\

(4500,5000) 
&30  & 33.47 & 1.69e-8 &22  
&21.62  & 5.54e-7  \\

(2500,5000) 
&58  & 49.97 & 1.82e-9 & 47 
& 34.35 & 3.14e-8  \\

(500,5000) 
& 959 & 556.41 & 6.03e-7 & 1114 
& 514.01 & 7.36e-10  \\

\bottomrule
\end{tabular}

\vspace{2mm}
\footnotesize

\end{table}

\subsection{Logistic Lasso}\label{sec:logistic-lasso-experiment}

The logistic Lasso problem is a well-established optimization model for binary classification. It can be formulated as a composite convex optimization problem of the form (\ref{eq:composite_convex_definition}). In contrast to the standard Lasso problem, the smooth component $h(x)$ in the logistic Lasso formulation is nonquadratic. To address this structure, we apply the CPN method (Algorithm~\ref{alg:Proximal_Newton}), as described in Section~\ref{sec:logistic_lasso_application}. Because the CPN method employs the NCN method as its inner solver, we specify the parameter settings used for the latter. In particular, the NCN method incorporates the adaptive regularization strategy in (\ref{eq:Adaptive_reg_rule_linesearch}), and the parameters of the HAN line search procedure are set to
\[
C_1 = 0.5, \qquad C_2 = 0, \qquad \beta = 0.5, \qquad M_1 = 2.
\]

To assess the computational efficiency of the proposed CPN method, we compare its performance with several well-established optimization methods from the literature, including FISTA\footnote{\url{https://github.com/Sina14KD/FISTA_Logistic_Lasso}}, SpaRSA\footnote{\textcolor{black}{We use the SpaRSA solver implemented in MATLAB's fitclinear function.}}, and PNOPT\footnote{\url{https://github.com/yuekai/PNOPT}}. These methods provide representative benchmarks for evaluating the effectiveness and computational performance of the CPN method on logistic Lasso problems.

\begin{table}[ht]

\centering
\caption{Numerical results for the logistic Lasso problem ($\mu=10^{-3}$)}
\label{tab:results-CPN-logistic}
\begin{tabular}{lcccccccc}
\toprule
\multirow{2}{*}{Case (m,n)}
& \multicolumn{4}{c}{Iterations (Res)}
& \multicolumn{4}{c}{CPU Time (s)} \\
\cmidrule(lr){2-5}
\cmidrule(lr){6-9}
& FISTA & SpaRSA & PNOPT & CPN
& FISTA & SpaRSA & PNOPT & CPN \\
\midrule

$(10^3,500)$ 
& 807 (T) & 98 (F) & 60 (T) & \text{8} (T)
& 1.36 & \text{0.73} & 1.92 & \text{1.46} \\

$(10^4,500)$ 
& 73 (T) & 22 (F) & 12 (T) & \text{7} (T)
& 1.99 & 1.37 & \text{0.82} & \text{0.96} \\

$(10^4,250)$ 
& 87 (T) & 16 (F) & 11 (F) & \text{6} (T)
& 1.16 & 0.76 & \text{0.48} & \text{0.37} \\

$(10^4,100)$ 
& 57 (T) & 13 (F) & 10 (F) & \text{6} (T)
& 0.45 & 0.76 & 0.40 & \text{0.23} \\

$(10^5,500)$ 
& 68 (T) & 11 (F) & 10 (F) & \text{6} (T)
& 17.1 & \text{1.28} & \text{1.81} & \text{2.97} \\

$(10^5,250)$ 
& 66 (T) & 10 (F) & 11 (T) & \text{6} (T)
& 8.07 & \text{1.05} & \text{1.21} & \text{1.07} \\

$(10^5,100)$
& 42 (T) & 9 (T) & 9 (T) & \text{5} (T)
& 2.38 & 0.83 & 0.62 & \text{0.45} \\

$(10^6,250)$ 
& 76 (T) & 11 (T) & 10 (T) & \text{6} (T)
& 90.9 & \text{4.41} & \text{8.13} & \text{9.76} \\

$(10^6,100)$ 
& 40 (T) & 8 (F) & 9 (F) & \text{5} (T)
& 56.7 & \text{1.89} & \text{3.17} & \text{3.04} \\

$(10^6,50)$ 
& 57 (T) & 11 (T) & 10 (T) & \text{6} (T)
& 13.5 & \text{1.68} & \text{2.24} & \text{1.90} \\

$(10^7,50)$ 
& 79 (T) & 11 (T) & 11 (T) & \text{7} (T)
& 179.1 & \text{10.5} & \text{23.4} & \text{23.8} \\

$(10^7,25)$ 
& 82 (T) & 11 (T) & 11 (T) & \text{6} (T)
& 87.7 & \text{7.09} & \text{11.4} & \text{11.0} \\

\textnormal{australian scale} 
& 806 (T) & 85 (F) & 31 (T) & \text{6} (T)
& 0.26 & 0.72 & 1.38 & \text{0.11} \\

covtype 
& 818 (T)& - & - & \text{8} (T)
& 144.2 & - & - & \text{8.82} \\

a1a test 
& 963 (T)& 158 (F)& 33 (T)& \text{7} (T)
& 8.33 & 0.92 & 2.34 & \text{0.90} \\

\bottomrule
\end{tabular}

\vspace{2mm}
\footnotesize
\textbf{Note:} Res is a Boolean variable, defined as \(T\) if \(r_{\text{pp}}<10^{-6}\) and \(F\) otherwise, where \(r_{\text{pp}}\) is defined in (\ref{eq:kkt_res_logistic_post_evaluation}). The symbol ``--'' indicates that the algorithm fails to find a minimizer.

\end{table}

Table~\ref{tab:results-CPN-logistic} presents a comparison of the four algorithms for the logistic lasso problem. For each case, we report the number of iterations and the CPU time required by each method.
Although all methods use a residual tolerance of $10^{-6}$, they employ different definitions of the residual. To provide a consistent comparison, we evaluate the final solutions using the \textit{post-processed} residual
\begin{equation}\label{eq:kkt_res_logistic_post_evaluation}
r_{\text{pp}} =
\frac{\left\|x-\textnormal{Prox}_{\frac{1}{L}g}
\left(x-\frac{1}{L}\nabla h(x)\right)\right\|_2}
{1+\|x\|_2},
\end{equation}
where $\mathcal{D}$ is the data matrix,
$L=\frac{1}{4N}\|\mathcal{D}\|_2^2$, and $h$ and $g$ are defined in~(\ref{eq:logistic_lasso_convex_composite}). In Table~\ref{tab:results-CPN-logistic}, T indicates that the post-processed residual satisfies $r_{\text{pp}}<10^{-6}$, while F indicates otherwise. For all cases except the last three, the data matrix $\mathcal{D}$ and the label vector $y^{\text{lr}}$ are generated randomly, whereas the last three cases are taken from the LIBSVM dataset collection~\cite{chang2011libsvm}. The regularization parameter is fixed at $\mu=10^{-3}$ for all cases.

The numerical results presented in Table~\ref{tab:lasso-well-known-algorithms-comparsion} show that the NCN algorithm performs well when the number of features is smaller than the number of samples. Since the CPN method employs the NCN algorithm as its inner solver, we therefore focus on cases in which the number of samples is larger than, or significantly larger than, the number of features. These cases, along with the performance results of the considered algorithms, are presented in Table~\ref{tab:results-CPN-logistic}. The numerical results indicate that CPN and FISTA are the only methods that consistently obtain high-accuracy solutions across all tested cases, as measured by the post-processed residual criterion $r_{\text{pp}}<10^{-6}$, where $r_{\text{pp}}$ is defined in (\ref{eq:kkt_res_logistic_post_evaluation}).  Among these two methods, CPN requires fewer iterations to attain the prescribed accuracy and exhibits faster convergence than FISTA. In general, among the four methods considered, CPN requires the fewest iterations to converge to the optimal solution. Also, in terms of CPU time, CPN demonstrates competitive performance compared with the SpaRSA and PNOPT methods. Overall, these results suggest that CPN can achieve high solution accuracy while maintaining a relatively low computational cost, as reflected in both the number of iterations and the CPU time.

\section{Conclusion}\label{sec:conclusion}

Newton-type methods can achieve fast local convergence by using second-order information. However, in different applications, including statistics and machine learning, the objective functions may be \(C^{1,1}\) or second-order nonsmooth, and classical Hessians may not exist. In this work, we proposed the NCN method, which uses a generalized Hessian based on the coderivative of the gradient mapping. The NCN method combines an adaptive regularization strategy with the proposed HAN line search to improve its global performance. We applied NCN to convex composite problems with a quadratic smooth component, including SVM and Lasso. The numerical results show that NCN outperforms the existing coderivative-based Newton method considered in our experiments and provides competitive performance compared with established first- and second-order methods. We also proposed the CPN method for convex composite problems with a nonquadratic smooth component. Numerical results for logistic Lasso show that CPN is competitive with existing first- and second-order methods. We further developed a general convergence framework for coderivative-based Newton methods. Under suitable assumptions, the framework establishes global convergence and superlinear local convergence. It also covers methods that use inexact second-order information. Several directions remain for future work. It would be useful to develop semistochastic and lazy variants of NCN to reduce the computational cost of second-order information while maintaining fast local convergence. Another direction is to study coderivative-based Newton methods when both first- and second-order information are inexact. Finally, extending the proposed methods and convergence framework to nonconvex problems is an interesting topic for future research.

\textbf{Code availability:} The code that supports the findings of this study is available from the corresponding author upon reasonable request.

\newpage
\appendix
\section{Appendix: The HAN line search method}\label{appendix:HAN}

Section~\ref{sec:HAN} presented the motivation and main ideas of the HAN line search. We now turn to its implementation and convergence analysis. The complete implementation is summarized in Algorithm~\ref{alg:HAN_linesearch2}. The algorithm is governed by two switching thresholds, $C_1$ and $C_2$, satisfying
$C_1 \ge C_2 \ge 0.$ Starting from the initial step size $\alpha_k=1$, the HAN line search adaptively switches among three line search strategies according to the magnitude of the current step size.

When $C_1\le \alpha_k\le 1$, the HAN line search behaves as the classical Armijo line search. Specifically, the trial step size is accepted if it satisfies
\begin{equation}\label{eq:Armijo_condition}
f(x_k+\alpha_k d_k)\le
f(x_k)+c\alpha_k\langle\nabla f(x_k),d_k\rangle,
\end{equation}
where $c\in(0,1)$. Otherwise, the step size is reduced according to $
\alpha_k\leftarrow\beta\alpha_k,$
where $0<\beta<1$, and the Armijo condition (\ref{eq:Armijo_condition}) is tested again. If no step size satisfying $\alpha_k>C_1$ fulfills the Armijo condition, the algorithm switches to the nonmonotone Armijo line search. In this stage, the step size continues to be reduced by the factor $\beta$ until it satisfies
\begin{equation}\label{eq:nonmonotone-condition-HAN}
f(x_k+\alpha_kd_k)
\le
F_k+c\alpha_k\langle\nabla f(x_k),d_k\rangle,
\end{equation}
where $
F_k:=\max_{j=0,1,\ldots,M}f(x_{k-j}), $
with $M=M_1$. If no step size satisfying $\alpha_k>C_2$ fulfills the above condition (\ref{eq:nonmonotone-condition-HAN}), the algorithm proceeds to the final stage.

In the final stage, the HAN line search switches to a Wolfe-type line search. The trial step size $\alpha_k$ is accepted if it satisfies \eqref{eq:nonmonotone-condition-HAN} with $M=M_2$, together with the curvature condition
\begin{equation}\label{eq:wolf_condition_HAN}
\langle\nabla f(x_k+\alpha_kd_k),d_k\rangle
\ge
c^\prime\langle\nabla f(x_k),d_k\rangle,
\end{equation}
where $c^\prime \in (c,1)$.

The HAN line search provides a unified framework that encompasses several classical line search strategies through appropriate choices of its algorithmic parameters. In particular, the Armijo line search is obtained by setting $C_1=C_2=0$, while the nonmonotone Armijo line search is recovered by choosing $C_1=1$ and $C_2=0$. Furthermore, when $C_1=C_2=1$, the HAN line search becomes a Wolfe-type line search, which reduces to the classical Wolfe line search for $M_2=0$ and to its nonmonotone extension for $M_2>0$.

The following lemma establishes sufficient conditions for the global convergence of optimization methods employing the HAN line search as their globalization strategy.

\begin{algorithm}
  \caption{HAN line search with Implementation Details}
  \label{alg:HAN_linesearch2}
\begin{algorithmic}[1]

\State \textbf{Input:} Given $d\in\mathbb{R}^n$, $C_1\geq C_2\geq0$, 
$M_1,M_2\geq0$, $S\geq0$, $c\in(0,1)$, and $c'\in(c,1)$.

\State \textbf{Initialize:} $\alpha=1$, $0<\beta<1$, $\gamma>1$, $s=0$,
and $\text{accepted}=\text{False}$.

\While{\text{accepted} = \text{False}}

    \If{$\alpha>C_1$}
        \If{$\alpha$ satisfies the Armijo condition~(\ref{eq:Armijo_condition})}
            \State \textbf{return} $\alpha$
        \Else
            \State $\alpha\leftarrow\beta\alpha$
        \EndIf

    \ElsIf{$\alpha>C_2$}
        \If{$\alpha$ satisfies the nonmonotone condition~(\ref{eq:nonmonotone-condition-HAN}) with $M=M_1$}
            \State \textbf{return} $\alpha$
        \Else
            \State $\alpha\leftarrow\beta\alpha$
        \EndIf

    \Else
        \If{$\alpha$ satisfies the nonmonotone condition~(\ref{eq:nonmonotone-condition-HAN}) with $M=M_2$}
    \If{$\alpha$ satisfies the curvature condition~(\ref{eq:wolf_condition_HAN})}
        \State \textbf{return} $\alpha$
    \ElsIf{$s<S$}
        \State $\alpha\leftarrow\gamma\alpha$
        \State $s\leftarrow s+1$
    \Else
        \State \textbf{return} $\alpha$
    \EndIf
\Else
    \State $\alpha\leftarrow\beta\alpha$
\EndIf
\EndIf
\EndWhile
\end{algorithmic}
\end{algorithm}

\begin{lemma}[\textbf{Global convergence under the HAN line search}]\label{lemma:HAN_convergence}
    Let $f:\mathbb{R}^n \to \mathbb{R}$ be of $C^{1,1}$ functions. Also, let $d_k$ be the search direction at iteration $k$ satisfy the following conditons:
    \begin{condition}\label{condition:decent_direction}
           Decent direction condition: $d_k$ satisfies $\langle -\nabla f(x_k), d_k \rangle >0$ when $\nabla f(x_k) \neq 0$. 
    \end{condition}
    \begin{condition}\label{condition:Angle_conditon}
        Angle condition: $d_k$ satisfies $\langle -\nabla f(x_k),d_k \rangle \ge \sigma_1\|\nabla f(x_k)\|^{\gamma}\|d_k\|^{\zeta}$ for constant $\sigma_1>0$, $\gamma > 0$, and $\zeta >0$.
    \end{condition}
    \begin{condition}\label{condition:bounded_direction}
        Bounded direction condition when $\nabla f(x_k) \neq 0$: Under the assumption that $\|\nabla f(x_k)\| > \epsilon$ for some $\epsilon > 0$ and for all $k \in \mathbb{N}$, $d_k$ satisfies $
\|d_k\| \le \sigma_2 \|\nabla f(x_k)\|,$ where $\sigma_2 >0$.

    \end{condition}

    Then, if we updating  $x_{k+1}=x_k+\alpha_k d_k$ where the step size $\alpha_k>0$ satisfies:
    \begin{condition}\label{condition:Nonmonotone_linesearch}
         Nonmonotone line search which is:
    \begin{equation}\label{eq:A4_condition_HAN_convergence}
        f(x_k+\alpha_kd_k) \le F_k+c \alpha_k \langle \nabla f(x_k),d_k \rangle, \textnormal{ for } c\in(0,1).
    \end{equation}
    where $F_k:=\max_{j=0,1,\dots,M} f(x_{k-j})$.
    \end{condition}

Then, $\lim_{k \to \infty} \|\nabla f(x_k)\|=0$.
   
\end{lemma}

\begin{proof}
    In the first step of the proof, we show that in each iteration there exists a positive step size (greater than a specific lower bound) satisfying (\ref{eq:A4_condition_HAN_convergence}), or
    \begin{equation}\label{eq:alpha_k_has_lower_bound}
        \alpha_k \ge \alpha_{\min}>0 \textnormal{ for } k\in \mathbb{N}.
    \end{equation}
   when $\|\nabla f(x_k)\|\neq0$ or $\|\nabla f(x_k)\|>\epsilon$. By defenition, $F_k:=\max_{j=0,1,\dots,M} f(x_{k-j})$, we use Lipschitz continuity of $\nabla f(x_k)$ (with constant $L$) to obtain the following estimate for the value of the function in the next iteration:
    \begin{equation}
    f(x_k+\alpha_k d_k) \le f(x_k)+\alpha_k \langle \nabla f(x_k),d_k \rangle+\frac{L}{2}\alpha_k^2\|d_k\|^2 \le F_k+\alpha_k \langle \nabla f(x_k),d_k \rangle+\frac{L}{2}\alpha_k^2\|d_k\|^2 .
    \end{equation}
    Now it is enough to show that there exists $\alpha_k>0$ such that the following inequality holds:
    \begin{equation} \label{eq:finding_alpha_k_in_A4_estimation_1}
        F_k+\alpha_k \langle \nabla f(x_k),d_k \rangle +\frac{L}{2}\alpha_k^2 \|d_k\|^2 \le F_k+c\alpha_k\langle \nabla f(x_k)
    ,d_k \rangle.
    \end{equation}
    Using (\ref{eq:finding_alpha_k_in_A4_estimation_1}), we find that $\alpha_k$ should satisfy the following inequality:
    \begin{equation}
        \alpha_k \le \frac{2(c-1)\langle \nabla f(x_k),d_k \rangle}{L\|d_k\|^2}. 
    \end{equation}
    Considering $- \| \nabla f(x_k) \| \|d_k\| \le \langle \nabla f(x_k) ,d_k \rangle$ and ($\mathcal{B}\ref{condition:bounded_direction}$); if we choose $\alpha_k \in (0, \frac{2(1-c)}{L \sigma_2}]$ we have: 
    \begin{equation}
        \alpha_k \le \frac{2(1-c)}{L \sigma_2}= \frac{2(1-c)\|\nabla f(x_k)\| \|d_k\|}{L \sigma_2 \|\nabla f(x_k)\|\|d_k\|}\le  \frac{2(c-1)\langle \nabla f(x_k),d_k \rangle}{L\|d_k\|^2}.
    \end{equation}
    Therefore, for each iteration, there exits $\alpha_k \ge \alpha_{\min}>0$ where $\alpha_{\min}:= \frac{2(1-c)}{L c^2}\beta$ (for $\beta \in (0,1)$ as backtraking paramer) such that (\ref{eq:finding_alpha_k_in_A4_estimation_1}) and therefore (\ref{eq:A4_condition_HAN_convergence}) are satisfied.

    In the next part, we back to the proof of the Final statement.  Using ($\mathcal{B}$\ref{condition:Nonmonotone_linesearch}) as well as ($\mathcal{B}$\ref{condition:Angle_conditon}), we have
    \begin{equation}\label{eq:HAN_convergence_gap_lower_bound}
        F_k-f(x_k+\alpha_kd_k)\ge c\alpha_k \langle -\nabla f(x_k),d_k \rangle \ge \sigma_1 c \alpha_k \|\nabla f(x_k)\|^{\gamma}\|d_k\|^{\zeta} \ge0. 
    \end{equation}
We know, when $\nabla f(x_k) \neq0$ , $f(x_{k+1}) < F_k$, so $\{F_k\}_{k\in \mathbb{N}}$ is a non-constant, monotone non-increasing sequence. Since $f$ is bounded below, by monotone convergence theorem we obtain $F_k \to f_{\inf}$; also, due to $f(x_k) \le F_k$ we have $\lim_{k \to \infty} f(x_k)=f_{\inf}$. Using (\ref{eq:HAN_convergence_gap_lower_bound}), we conclude
\begin{equation}\label{eq:HAN_convergence_limit_gap_bound}
\lim_{k \to \infty} \sigma_1 c \alpha_k \|\nabla f(x_k)\|^{\gamma}\|d_k\|^{\zeta}=0.
\end{equation}
To show that (\ref{eq:HAN_convergence_limit_gap_bound}) indicates $\lim_{k \to \infty} \|\nabla f(x_k)\|=0$, we use condratiction. Suppose that $\|\nabla f(x_k)\|>\epsilon$ for all $k\in \mathbb{N}$. By ($\mathcal{B}$\ref{condition:decent_direction}) we get that $\|d_k\|>\hat{\epsilon}$ for some $\hat{\epsilon}>0$. So, (\ref{eq:HAN_convergence_limit_gap_bound}) implies that $\lim_{k \to \infty} \alpha_k=0$. This is contradition to  (\ref{eq:alpha_k_has_lower_bound}) which has already been proved, and therefore $\lim_{k \to \infty} \|\nabla f(x_k)\|=0$.

\end{proof}

More generally, the preceding analysis provides a unified convergence framework that is independent of the particular optimization algorithm. Specifically, any descent method whose search directions satisfy Conditions $\mathcal{B}$\ref{condition:decent_direction}--\ref{condition:bounded_direction} and whose step sizes are generated by the HAN line search satisfying Condition $\mathcal{B}$\ref{condition:Nonmonotone_linesearch} is globally convergent. This demonstrates that the HAN line search is not tied to the proposed Newton-type method, but can be employed as a general globalization strategy for a broad class of descent algorithms.

\section{Appendix: Required Propositions}

This appendix presents results concerning the acceptance of unit step sizes in the HAN line search. These results are used to establish the local convergence properties of the NCN method.

\begin{proposition}
\label{prop:unit_stepsize_semismooth}
Suppose that a function $f : \mathbb{R}^n \to \mathbb{R}$ is $C^1$-smooth around $\bar{x}$ with $\nabla f(\bar{x}) = 0$, and that $\nabla f$ is semismooth at this point. Let a sequence $\{x_k\}$ converge to $\bar{x}$ with $x_k \neq \bar{x}$ for all $k \in \mathbb{N}$, and let a sequence $\{d_k\}$ satisfy
\begin{equation}
\|x_k + d_k - \bar{x} \|= o(\|x_k - \bar{x}\|).
\end{equation}
Suppose further that there exists $\kappa > 0$ such that
\[
\langle \nabla f(x_k), d_k \rangle \le -\kappa^{-1} \|d_k\|^2+o(\|d_k\|^2)
\]
for all sufficiently large $k \in \mathbb{N}$. Then, for $F_k:=\max_{j=0,1,\dots,M} f(x_{k-j})$ and $c \in (0,1/2)$, we have
\[
f(x_k + d_k) \le F_k + c \langle \nabla f(x_k), d_k \rangle
\]
for all sufficiently large $k$.
\end{proposition}

\begin{proof}
    It follows from the Theorem 3.3 of \cite{facchinei1995minimization} and the fact that $F_k \ge f(x_k)$.
    
\end{proof}

\begin{proposition}\label{prop:super_linear_local_rate_curvuture_condition}
    Under the setting of the Proposition~\ref{prop:superlinear_convergent_rate}, eventually unite step size passes the curvuture condition (\ref{eq:wolf_condition_HAN}). 
\end{proposition}

\begin{proof}
First we consider the assumption (i) and show that eventually the step size $\alpha_k=1$ passes the condition (\ref{eq:wolf_condition_HAN}). To do that, by $\text{semismooth}^*$ property and assumption (i), we have the semismoothness of $\nabla f$ at $\bar{x}$ which gives us
\[
\nabla f(x_k+d_k)=\nabla f(\bar{x})+A_k(x_k+d_k-\bar{x})+o(\|x_k+d_k-\bar{x}\|),
\]
where $A_k \in \partial_C^2 f(x_k+d_k):=\textnormal{co}\overline{\nabla}(\nabla f(x_k+d_k))$. By (\ref{eq:superlineae_local_rate_last_equation}), (\ref{eq:super_linear_convergence_nonmonotone_direction_error_limit_1}), and Proposition 2.6.2(d) in \cite{clarke2023optimization} (which indicate  that the matrix sequence $\{H_k\}$ is bounded), we obtain $\nabla f(x_k+d_k)=o(\|x_k-\bar{x}\|)$ and $\nabla f(x_k+d_k)=o(\|d_k\|)$. In addition, using ($\ref{eq:super_linear_convergence_nonmonotone_direction_estimate}$), $\|\epsilon_k\|=o(\|x_k -\bar{x}\|)$, and ($\ref{eq:super_linear_convergence_nonmonotone_direction_error_limit_1}$) we have

\begin{equation}\label{eq:angle_estimation_with_direction_norm}
    \langle \nabla f(x_k),-d_k \rangle \ge \kappa^{-1}\|d_k\|^2+\langle \epsilon_k,-d_k \rangle =\kappa^{-1}\|d_k\|^2+o(\|d_k\|^2).
\end{equation}
Hence, for large $k  \in \mathbb{N}$, we have
\begin{equation}
\begin{split}
&    \langle \nabla f(x_k+d_k), d_k \rangle- c^\prime \langle \nabla f(x_k),d_k \rangle \\
&=- c^\prime \langle \nabla f(x_k),d_k \rangle+o (\|d_k\|^2) \ge \frac{c^\prime}{\kappa}\|d_k\|^2+o(\|d_k\|^2)+o(\|d_k\|^2)>0.
\end{split}
\end{equation}

Next, we consider the assumption (ii). Based on Claim 1 in Proposition \ref{prop:superlinear_convergent_rate}, $\bar{x}$ is a tilt-stable minimizer. Due to this fact, by applying (\ref{eq:second_ordr_condition_of_the_tilt_stable_minimizer}) from Proposition \ref{prop:tilt_stability_second_order_growth}, we obtain the following inequality  
\begin{equation}\label{eq:second_order_condiiton_tilt_stable_minimizer_2}
    \langle \nabla f(x)-\nabla f(u) \rangle \ge \frac{1}{2 \kappa}\|x-u\|^2.
\end{equation}
By $c^\prime >1-\frac{1}{l \kappa}$ and ($\ref{eq:super_linear_convergence_nonmonotone_direction_error_limit_1}$), 
\begin{equation}
    \begin{split}
        & \langle \nabla f(x_k+d_k),-d_k \rangle - c^\prime \langle \nabla f(x_k),- d_k \rangle\\
        & \langle \nabla f(x_k),- d_k \rangle - \frac{1}{\kappa} \|d_k\|^2 -c^\prime \langle \nabla f(x_k),-d_k \rangle \\
        & \le \|d_k\|^2 \Big(-\frac{1}{\kappa}+(1-c^\prime)l\frac{\|x_k-\bar{x}\|}{\|d_k\|}\Big)<0,
    \end{split}
\end{equation}
where the first inequality obtained from (\ref{eq:second_order_condiiton_tilt_stable_minimizer_2}). The proof is complete.

\end{proof}

\newpage
\printbibliography

\end{document}